\documentclass{amsart}
\usepackage{amsmath}
\usepackage{amssymb}
\usepackage[abbrev]{amsrefs}
\usepackage{graphicx,epsf,amsmath}  
\usepackage{epsf,graphicx}
\usepackage{comment}
\usepackage{color}

\newtheorem{theorem}{Theorem}[section]
\newtheorem{lemma}[theorem]{Lemma}
\newtheorem{prop}[theorem]{Proposition}

\newtheorem{definition}[theorem]{Definition}

\newtheorem*{theorem*}{Theorem}

\theoremstyle{definition}
\newtheorem*{remark}{Remark}

\newcommand{\tv}{{\tilde{v}}}

\newcommand{\Rn}{\mathbb{R}^n}
\newcommand{\ep}{\varepsilon}

\newcommand{\R}{\mathbb{R}}
\newcommand{\C}{\mathcal{C}}

\newcommand{\be}{\begin{equation}}
\newcommand{\ee}{\end{equation}}
\newcommand{\bee}{\begin{equation*}}
\newcommand{\eee}{\end{equation*}}
\newcommand{\bea}{\begin{eqnarray}}
\newcommand{\eea}{\end{eqnarray}}
\newcommand{\bs}{\begin{split}}
\newcommand{\es}{\end{split}}

\numberwithin{equation}{section}

\title[]{Regularity of radial extremal solutions for some non local semilinear equations }
\author[A. Capella]{Antonio Capella}
\address{\noindent  A. Capella - Instituto de Matem\'aticas, Universidad Nacional 
Aut\'onoma de M\'exico, 
Circuito Exterior, Ciudad Universitaria	
 C.P. 04510, M\'exico D.F., M\'exico.}\email{capella@matem.unam.mx}

\author[J. Davila]{Juan D\'avila}
\address{\noindent J. D\'avila- Departamento de
Ingenier\'{\i}a  Matem\'atica and CMM, Universidad de Chile, Casilla
170 Correo 3, Santiago, Chile.} \email{jdavila@dim.uchile.cl}

\author[L. Dupaigne]{Louis Dupaigne}
\address{L. Dupaigne- LAMFA, UMR CNRS 6140, Universit´e Picardie Jules Verne,
33 rue St Leu, 80039 Amiens, France.}
\email{ldupaigne@math.cnrs.fr}

\author[Y. Sire]{Yannick Sire}
\address{Y. Sire- Universit\'e Aix-Marseille 3, Paul C\'ezanne -- LATP  Marseille, France.}
\email{sire@cmi.univ-mrs.fr}

\begin{document}
\begin{abstract}
We investigate stable solutions of elliptic equations of the type
\begin{equation*}
\left \{
\begin{aligned}
(-\Delta)^s u&=\lambda f(u) \qquad
{\mbox{ in $B_1 \subset \R^{n}$}}
\\
u&= 0
\qquad{\mbox{ on $\partial B_1$,}}\end{aligned}\right . \end{equation*}
where $n\ge2$, $s \in (0,1)$, $\lambda \geq 0$ and $f$ is any smooth positive superlinear function. The operator $(-\Delta)^s$ stands for the fractional
Laplacian, a pseudo-differential operator of order $2s$.  According to the value of $\lambda$, we study the existence and regularity of weak solutions $u$.

\end{abstract}
\maketitle



\bigskip


\section{Introduction}
We are interested in the regularity properties of stable solutions
satisfying the following semilinear problem involving the fractional Laplacian
\begin{equation}
\label{problem}
\left\{
\begin{array}{ll}
(-\Delta)^{s}u = \lambda f(u)&\text{in }  B_1,
\\
u=0&\text{on } \partial B_1.
\end{array}
\right.
\end{equation}
Here, $B_1$ denotes the unit-ball  in $\mathbb{R}^n$, $n\ge2$, and $s \in (0,1)$. The operator  $(-\Delta)^{s}$ is defined as follows: let $\left \{ \varphi_k \right \}_{k=1}^\infty$ denote an orthonormal basis of $L^2(B_{1})$ consisting of eigenfunctions of $-\Delta$ in $B_1$ with homogeneous Dirichlet boundary conditions, associated to the eigenvalues $\{\mu_{k}\}_{k=1}^{\infty}$. Namely, $0<\mu_{1}<\mu_{2}\le\mu_{3}\le\dots\le\mu_{k}\to+\infty$, $\int_{B_{1}}\varphi_{j}\varphi_{k}\;dx=\delta_{j,k}$ and
\begin{equation*}
\left\{
\begin{array}{ll}
-\Delta \varphi_k =\mu_k \varphi_k&\text{in }  B_1
\\
\varphi_k=0&\text{on } \partial B_1.
\end{array}
\right.
\end{equation*}
The operator $(-\Delta)^s$ is defined for any $u\in C^\infty_{c}(B_{1})$ by
$$
(-\Delta)^s u=\sum_{k=1}^\infty  \mu_k^s u_k\varphi_k ,
$$
where
 $$
u=\sum_{k=1}^\infty u_k \varphi_k,\qquad\text{and}\qquad  u_{k}=\int_{B_{1}}u\varphi_{k}\;dx.
$$
This operator can be extended by density for $u$ in the Hilbert space \begin{equation} \label{defH}
H=\{u\in L^2(B_{1})\; : \; \| u \|^2_{H} = \sum_{k=1}^\infty \mu_k^{s}\vert u_k\vert^2 <+\infty\}.
\end{equation}
Note that
$$
H=
\left\{
\begin{aligned}
H^s(B_{1})&\qquad\text{if $s\in (0,1/2)$},\\
H^{1/2}_{00}(B_{1})&\qquad\text{if $s=1/2$},\\
H^s_{0}(B_{1})&\qquad\text{if $s\in (1/2,1)$},
\end{aligned}
\right.
$$
see Section \ref{prelim} for further details. In all cases,
$
(-\Delta)^s: H\to H'
$
is an isometric isomorphism from $H$ to its topological dual $H'$. We denote by $(-\Delta)^{-s}$ its inverse, i.e. for $\psi~\in~ H'$, $\varphi=(-\Delta)^{-s}\psi$ if $\varphi$ is the unique solution in $H$ of $(-\Delta)^{s}\varphi=\psi$.

We will assume that  the nonlinearity $f$ is smooth, nondecreasing, 
\begin{equation}\label{Hyp_f}
 f(0)>0,\quad\text{ and }\quad\lim_{u\to +\infty} \frac{f(u)}{u}=+\infty.
 \end{equation}
 In the spirit of  \cite{bcmr}, weak solutions for \eqref{problem} are defined as follows: let 
 $\varphi_{1}>0$ denote the eigenfunction associated to the principal eigenvalue of the operator $-\Delta$ with homogeneous Dirichlet boundary condition on $B_{1}$, normalized by $\| \varphi_{1}\|_{L^2(B_{1})}~=~1$.
 \begin{definition}
A measurable function $u$ in $B_{1}$ such that $\int_{B_{1}}\vert u\vert \varphi_{1}\;dx<+\infty$ and $\int_{B_{1}}f(u) \varphi_{1}\;dx<+\infty$,
is a weak solution of \eqref{problem} if
\begin{equation} \label{j1}
\int_{B_1} u\psi\;dx =\lambda \int_{B_1} f(u)(-\Delta)^{-s}\psi\ dx,
\end{equation}
for all $\psi\in C^{\infty}_{c}(B_{1})$.
\end{definition}
The right-hand side in \eqref{j1} is well defined, since for every $\psi\in C^\infty_{c}(B_{1})$, there exists a constant $C>0$ such that $\vert (-\Delta)^{-s}\psi\vert\le C\varphi_{1}$, see Lemma \ref{lemma1bcmr} and its proof.

We shall be interested in weak solutions of \eqref{problem} having the following stability property.
 \begin{definition}\label{stabFrac}
A weak solution $u$ of \eqref{problem} is semi-stable if for all { $\psi \in C^\infty_c(B_{1})$} we have
\begin{equation} \label{stability} 
\int_{B_1} |(-\Delta)^{\frac{s}{2}} \psi|^2\;dx \ge \int_{B_1} f'(u) \psi^2\;dx. 
\end{equation} 
\end{definition}
The following result gives the existence of solutions according to the values of $\lambda$.
\begin{prop} 
\label{exist}
Let $s \in (0,1)$. There exists $\lambda^*> 0$ such that
\begin{itemize}
\item for $0<\lambda<\lambda^*$, there exists a minimal solution $u_{\lambda}\in H\cap L^\infty(B_{1})$ of \eqref{problem}. In addition, $u_{\lambda}$ is semi-stable and increasing with $\lambda$.
\item for $\lambda=\lambda^*$, the function $u^*=\lim_{\lambda\nearrow\lambda^*}u_{\lambda}$ is a weak solution of \eqref{problem}. We call $\lambda^*$ the extremal value of the parameter and $u^*$ the extremal solution.
\item for $\lambda>\lambda^*$, \eqref{problem} has no solution $u\in H\cap L^\infty(B_{1})$.
\end{itemize}
\end{prop}
For the proof, see Section \ref{existence}.  
\begin{remark}
Proposition \ref{exist}  remains true when $B_{1}$ is replaced by any smoothly bounded domain.
\end{remark}

\begin{remark} For $0<\lambda<\lambda^*$, the solution $u_{\lambda}$ is minimal in the sense that $u_{\lambda}\le u$ for any other weak solution $u$. In particular, 
$u_{\lambda}$ and $u^*$ are radial. In addition,
 $u_{\lambda}$ and $u^*$ are radially decreasing (see Section \ref{sec gnn}) and $u_{\lambda}~\in~C^\infty(B_{1})\cap C^{\alpha}(\overline{B_{1}})$ for  $\alpha \in (0,\min(2s,1))$ (see Section \ref{prelim}). If $u^*$ is bounded, then we also have $u^*~\in~C^\infty(B_{1})\cap C^{\alpha}(\overline{B_{1}})$ for  $\alpha \in (0,\min(2s,1))$, using again Section \ref{prelim}. 
\end{remark}

Here is our main result, concerning the regularity of the extremal solution $u^*$.

 \begin{theorem}\label{thm_extremal}
Assume $n\ge 2$ and let $u^*$ be the extremal
solution of \eqref{problem}. We have that:

\hspace{0.5cm} (a) If $n < 2 ( s + 2 + \sqrt{2(s+1)}) $ then $u^*\in L^\infty(B_1)$.

\hspace{0.5cm} (b) If $ n\ge 2 ( s + 2 + \sqrt{2(s+1)})$, then for any $\mu<n/2-1-\sqrt{n-1}-s$, 
there exists a constant $C>0$ such that $u^*(x)\le C\vert x\vert^{-\mu}$ for all $x\in B_1$.
\end{theorem}

\begin{remark}
In particular, for any $2\le n\le6$, any $s\in(0,1)$, and any smooth nondecreasing $f$ such that \eqref{Hyp_f} holds, the extremal solution is always bounded. 
\end{remark}

\begin{remark}
We do not know if the bound $n<2(s+2+\sqrt{2(s+1)})$ is optimal for the regularity of $u^*$. We note however that $\lim_{s\to 1^-}2(s+2+\sqrt{2(s+1)})=10$, and that the extremal solution of 
\begin{align}\label{gelfand2} 
\left\{
\begin{aligned}
-\Delta u &= \lambda f(u) && \hbox{in $\Omega$}
\\
u &= 0 && \hbox{on $\partial \Omega$}.
\end{aligned}
\right.
\end{align} 
is singular when $\Omega=B_{1}$, $f(u)=e^{u}$, and $n=10$ (see e.g. \cite{joseph-lundgren}).
\end{remark}

Nonlinear equations involving fractional powers of the Laplacian are currently actively studied. 
Caffarelli, Salsa and Silvestre studied free boundary problems for such operators in \cites{cafS,caffarelli-salsa-silvestre}.
Cabr\'e and Tan \cite{cabre-tan} obtained several results in analogy with the classical Lane-Emden problem $-\Delta u = u^p$, posed on bounded domains and entire space, such as the role of the critical exponent. Previously, some authors considered elliptic equations with nonlinear Neumann boundary condition, which share some properties with semilinear equations of the form \eqref{problem}, see e.g. \cite{cabre-solamorales,davila-dupaigne-montenegro}.

Equation \eqref{problem} is the fractional Laplacian version of the classical semilinear elliptic equation \eqref{gelfand2}. 
When $f(u) = e^u$, \eqref{gelfand2} is known as the Liouville equation \cite{liouville} or the Gelfand problem \cite{gelfand}.
Joseph and Lundgren \cite{joseph-lundgren} showed in this case
 that if $\Omega$ is a ball, then the extremal solution $u^*$ of
\eqref{gelfand2}
is bounded if and
only if $n<10$. Crandall and Rabinowitz \cite{crandall-rabinowitz} and Mignot and Puel \cite{mignot-puel} proved that if
$f(u)=e^u$ and $n<10$ then for { any} smoothly bounded
domain $\Omega$, $u^*$ is bounded. Using Hardy's inequality, Brezis and V\'azquez \cite{brezis-vazquez}
provided a different proof that $u^*$ is singular when $\Omega=B_1$ and $n \ge 10$.
For some other explicit nonlinearities, such as $f(u)=(1+u)^p$ with $p>1$ or $p<0$,
the critical dimension for the regularity of the extremal solution is known 
(for further details see the above mentioned references).
For general nonlinearities, Nedev \cite{nedev} proved that for any convex function $f$
satisfying \eqref{Hyp_f}, and any smooth bounded
domain  $\Omega\subset\R^n$, $n \le 3$, $u^*$ is bounded. This result has
been extended by Cabr\'e to the case $n=4$ and $\Omega$ strictly
convex \cite{cabre}. Finally, Cabr\'e and Capella
\cite{CC04} showed that if $\Omega$ is a ball and $n\le
9$ then for any nonlinearity $f$ satisfying \eqref{Hyp_f},
the extremal solution is bounded. 


%



\section{Preliminaries} \label{prelim}
\subsection{Functional spaces}
We start by recalling some functional spaces, see for instance \cite{lions-magenes,tartar}.
For $s \ge 0$, $H^s(\R^n)$ is defined as
$$
H^s(\R^n) = \{ u \in L^2(\R^n): |\xi|^s \hat u(\xi) \in L^2(\R^n) \}
$$
where $\hat u$ denotes the Fourier transform of $u$, with norm
$$
\| u \|_{H^s(\R^n)} = \| (1+ |\xi|^s) \hat u(\xi) \|_{L^2(\R^n)} .
$$
This norm is equivalent to
$$
\| u \|_{L^2(\R^n)} + \left( \int_{\R^n} \int_{\R^n} \frac{|u(x)-u(y)|^2}{|x-y|^{n+2s}} d x \, d y \right)^{1/2} .
$$

Given a smooth bounded domain $\Omega\subset\R^n$ and $0<s<1$, the space $H^s(\Omega)$ is defined as the set of functions $u \in L^2(\Omega)$ for which the following norm is finite
$$
\| u \|_{H^s(\Omega)} = \| u \|_{L^2(\Omega)} +
\left(
\int_{\Omega} \int_{\Omega} \frac{|u(x)-u(y)|^2}{|x-y|^{n+2s}} d x \, d y
\right)^{1/2} .
$$
An equivalent construction consists of restrictions of functions in $H^s(\R^n)$.
We define $H_0^s(\Omega)$ as the closure of $C_c^\infty(\Omega)$ with respect to the norm $\| \cdot \|_{H^s(\Omega)} $.
It is well known that for $0<s\le\frac{1}{2}$, $H_0^s(\Omega) = H^s(\Omega)$, while for $1/2<s<1$ the inclusion $H_0^s(\Omega) \subseteq H^s(\Omega)$ is strict (see Theorem~11.1 in
\cite{lions-magenes}).

The space $H$ defined in \eqref{defH} is the interpolation space $(H_0^2(\Omega),L^2(\Omega))_{s,2}$,
see for example \cite{lions-magenes,adams,tartar}.
Here we follow the notation from \cite[Chap. 22]{tartar}.
J.-L. Lions and E. Magenes \cite{lions-magenes} showed that  $(H_0^2(\Omega),L^2(\Omega))_{s,2} = H_0^s(\Omega)$ for $0<s<1$, $s\not=1/2$, while
$$
(H_0^2(\Omega),L^2(\Omega))_{1/2,2} = H_{00}^{1/2}(\Omega)
$$
where
$$H_{00}^{1/2}(\Omega) = \{u\in H^{1/2}(\Omega) \; : \; \int_{\Omega}\frac{u(x)^2}{d(x)}\;dx<+\infty\},$$
and $d(x) = {\rm dist}(x,\partial \Omega)$ for all $x\in \Omega$.

An important feature of the operator $(-\Delta)^s$ is its nonlocal character, which is best seen by realizing the fractional Laplacian as the boundary operator of a suitable extension in the half-cylinder $\Omega\times (0,\infty)$. Such an interpretation was demonstrated in \cite{cafS} for the fractional Laplacian in $\R^n$. Their construction can easily be extended to the case of bounded domains as described below.

Let us define
\begin{align*}
\C &= \Omega\times(0,+\infty),\\
\partial_L \C&=\partial \Omega \times [0,+\infty ).
\end{align*}
We write points in the cylinder using the notation $(x,y)\in \C=\Omega \times (0,+\infty)$.

Given $s\in(0,1)$, 
consider the space $H^1_{0,L}(y^{1-2s})$ of measurable functions $v:\C \to \R$ such that $v\in H^1(\Omega\times(s,t))$ for all $0<s<t<+\infty$, $v=0$ on $\partial_L \C$ and for which the following norm is finite
\begin{eqnarray*}
\| v \|_{H_{0,L}^1(y^{1-2s})}^2= \int_{\C}y^{1-2s}\left| \nabla v \right|^2 \; dxdy.
\end{eqnarray*}

\begin{prop} \label{trace theorem} There exists a trace operator from $H_{0,L}^1(y^{1-2s})$ into $H_0^s(\Omega)$. Furthermore, the space $H$ given by \eqref{defH} is characterized by
\begin{align*}
H &=
\{u = tr_{\Omega}v \; : \; v\in H_{0,L}^1(y^{1-2s})\} .
\end{align*}
\end{prop}
\proof
For the case $s=1/2$ see Proposition 2.1 in \cite{cabre-tan}.

We consider now $s\neq1/2$.
Restating the results of Paragraph 5 of J.-L.~Lions  \cite{lions}, there exists a constant $C>0$ such that
$$
\| v(\cdot,0) \|_{H^s(\R^n)}^2 \le C \int_{\R^n \times (0,+\infty) }y^{1-2s}\left(v^2+\left| \nabla v \right|^2\right) \; dxdy,
$$
whenever the right-hand side in the above inequality is finite. Now for any $v\in H_{0,L}^1(y^{1-2s})$,
$$
\int_{\C}y^{1-2s} v^2 \; dxdy\le C\int_{\C}y^{1-2s}\left| \nabla v \right|^2 \; dxdy,
$$
as follows from the standard Poincar\'e inequality in $\Omega$. Hence, extending $v$ by zero outside $\C$, we deduce that
$$
\| v(\cdot,0) \|_{H^s(\Omega)} \le C \| v \|_{H_{0,L}^1(y^{1-2s})}.
$$
This inequality shows that there exists a linear bounded trace operator
$$
tr_{\Omega}:H_{0,L}^1(y^{1-2s})\to H^s(\Omega).
$$
This operator has its image contained in $H_0^s(\Omega)$. This is direct for $0<s<1/2$ because in this case $H_0^s(\Omega)= H^s(\Omega)$. If $1/2<s<1$ we argue that any $v \in H_{0,L}^1(y^{1-2s})$ can be approximated by functions in $H_{0,L}^1(y^{1-2s})$ that have support away from $\partial_L \C$.
The trace of any such function has compact support in $\Omega$ and is therefore in $H_0^s(\Omega)$.
In all cases, this implies that the image of the trace operator is contained in $H$.

Let us prove $tr_{\Omega}: H_{0,L}^1(y^{1-2s}) \to H $ is surjective. Take a function $u\in H$ and let us prove that there exists $v\in H_{0,L}^1(y^{1-2s})$ such that $tr_{\Omega}(v)=u$. Write its spectral decomposition
$
u(x)=\sum_{k=1}^{+\infty} b_{k}\varphi_{k}(x)
$
and consider the function
\begin{equation} \label{juan1}
v(x,y) = \sum_{k=1}^{+\infty}b_{k}\varphi_{k}(x)g_{k}(y),
\end{equation}
where $g_{k}$ satisfies
\begin{align}
\label{bessel ode}
& g_{k}'' + \frac{1-2s}{y} g_{k}' - \mu_{k}g_{k} = 0\quad\text{in $(0,+\infty)$}
\\
\label{bessel ode 2} 
& g_{k}(0)=1 \qquad
g_{k}(+\infty)=0.
\end{align}
This ODE is a Bessel equation. Two independent solutions are given by $y^s I_{s}(\sqrt{\mu_{k}}y)$ and $y^s K_{s}(\sqrt{\mu_{k}}y)$, where $I_{s}, K_{s}$ are the modified Bessel functions of the first and second kind, see \cite{abramowitz-stegun}. Since $I_{s}$ increases exponentially at infinity and $K_{s}$ decreases exponentially, the solution we are seeking has the form
$$
g_{k}(y) = c_{k}y^s K_{s}(\sqrt{\mu_{k}}y).
$$
It is well-known that $K_{s}(t) = a t^{-s} +o(t^{-s})$ as $t\to0$, where $a>0$. Therefore, one can choose $c_{k}$ such that $g_{k}(0)=1$ and one can see that $g_{k}$ can be written in the form
$$
g_{k}(y) = h(\sqrt{\mu_{k}}y),
$$
for a fixed function $h$ that verifies $h(0)=1$ and $h'(t)=-c t^{2s-1}+o(t)$ as $t\to 0$, for some constant $c=c_{n,s}>0$ depending only on $s$ and $n$. This implies that
\begin{align}
\label{formula g p}
\lim_{y\to 0^+} -y^{1-2s}g_{k}'(y) = c_{n,s} \mu_k^{s}.
\end{align}
Since each of the functions $g_k$ decreases exponentially at infinity we see that $v$
defined by \eqref{juan1} is smooth for $y>0$, $x\in \Omega$ and moreover satisfies
$$
{\rm div}\, (y^{1-2s} \nabla v)=0 \qquad
{\mbox{ in $\C$}} .
$$
Let us check that $v\in H_{0,L}^1(y^{1-2s})$.
For any $y>0$, by the properties of $\varphi_k$:
$$
\int_{\Omega} |\nabla v(x,y)|^2 \, d x =
\sum_{k=1}^{\infty} b_k^2 ( \mu_k g_k(y)^2 + g_k'(y)^2)
$$
Integrating with respect to $y$ over $(\delta,+\infty)$ where $\delta>0$:
\begin{align}
\label{int delta}
\int_\delta^\infty \int_{\Omega} y^{1-2s}|\nabla v(x,y)|^2 \, d x d y
= \sum_{k=1}^{\infty} b_k^2 (- y^{1-2s} g_k'(y) g_k(y) )\Big|_{y=\delta}.
\end{align}
From the ODE \eqref{bessel ode} we deduce that $g_k\ge 0$, $g_k'\le 0$ and $g_k'(y) y^{1-2s}$ is non-decreasing. Thus, if $\delta_i \downarrow 0$, $i\to\infty$ is a decreasing sequence,
$- \delta_i^{1-2s} g_k'(\delta_i) g_k(\delta_i) $ is increasing. By monotone convergence and thanks to \eqref{formula g p} we deduce
$$
\int_0^\infty \int_{\R^n} y^{1-2s} |\nabla v(x,y)|^2 \, d x d y = c_{n,s}\sum_{k=1}^{+\infty} b_{k}^2\mu_{k}^s.
$$
This proves that $H \subseteq tr_{\Omega}( H_{0,L}^1(y^{1-2s}) ) $.
\qed

\bigskip

Let us remark that if $u \in H$, then
the minimization problem
$$
\min_{} \left\{
\int_{\C} y^{1-2 s} |\nabla v|^2 \;dxdy:
v \in H_{0,L}^1(y^{1-2s}), \  tr_{\Omega}(v) = u
\right\}
$$
has a solution $v \in H_{0,L}^1(y^{1-2s})$, by the weak lower semi-continuity of the norm $\| \ \|_{ H_{0,L}^1(y^{1-2s})}$ and continuity of $tr_{\Omega}$. Moreover the minimizer $v$ is unique, which follows e.g.\@ from the strict convexity of the functional. By standard elliptic theory $v(x,y)$ is smooth for $y>0$ and satisfies
\begin{align*}
\left\{
\begin{aligned}
{\rm div}\, (y^{1-2s} \nabla v) & =0 &&
\mbox{in $\C$}
\\
v&=0 && \mbox{on } \partial_L \C
\\
v&=u && \mbox{on } \Omega \times \{ 0 \}
\end{aligned}
\right.
\end{align*}
where the boundary condition on $\Omega \times \{ 0 \}$ is in the sense of trace.
For each $y>0$ we may write $v(x,y) = \sum_{k=1}^\infty \varphi_k(x) g_k(y)$
where $g_k(y) = \int_{\Omega} v(x,y) \, d x$.
Since $v(\cdot,y) \to u$ in $L^2(\Omega)$ as $y\to 0$,
$g_k(0)$ are the Fourier coefficients of $u$, that is $u = \sum_{k=1}^\infty g_k(0) \varphi_k$.
Then we deduce that $g_k(y)$ is smooth for $y>0$ and satisfies the ODE \eqref{bessel ode}. One can check that $g_k(y) \to 0 $ as $y\to+\infty$ and therefore $g_k(y) = c_k y^s K_s(\sqrt{\mu_k} y)$ for all $y>0$ and some $c_k \in \R$. Then, similarly as in \eqref{int delta},  we obtain for $\delta>0$
\begin{align}
\label{int delta 2}
\int_\delta^\infty \int_{\R^n} y^{1-2s}|\nabla v(x,y)|^2 \, d x d y
= \sum_{k=1}^{\infty}  (- y^{1-2s} g_k'(y) g_k(y) )\Big|_{y=\delta}.
\end{align}
Arguing as before, for each $k$
$$
\lim_{y\downarrow0}  (- y^{1-2s} g_k'(y) g_k(y) ) = c \mu_k^s g_k(0)^2 .
$$
We deduce from \eqref{int delta 2} that
$$
\|u\|_H^2 = \sum_{k=1}^\infty \mu_k^s g_k(0) = c \| v \|_{H_{0,L}^1(y^{1-2s})}^2.
$$
In what follows we will call $v$ the {\em canonical extension} of $u$.

\subsection{Solvability for data in $H^{-s}(\Omega)$}


This section is devoted to prove the following lemma:
\begin{lemma}
Let $h \in H'$. Then, there is a unique $u \in H$ which solves
\begin{equation} \label{linear problem} 
\left\{
\begin{array}{ll}
(-\Delta)^{s}u = h&\text{in }  \Omega
\\
u=0&\text{on } \partial \Omega.
\end{array}
\right.
\end{equation}
Moreover $u$ is the trace of $v \in H_{0,L}^1(y^{1-2s})$,	 where $v$ is the unique solution to
\begin{align}
\label{extension}
\left\{
\begin{aligned}
{\rm div}\, (y^{1-2s} \nabla v) & =0 &&
\mbox{in $\C$}
\\
v&=0 && \mbox{on } \partial_L \C
\\
- y^{1-2s} v_y&=c_{n,s} h && \mbox{on } \Omega \times \{ 0 \}
\end{aligned}
\right.
\end{align}
where $c_{n,s}>0$ is a constant depending on $n$ and $s$ only.
\end{lemma}
\begin{remark}
Equation  \eqref{extension} is understood in the sense that 
$v\in H_{0,L}^1(y^{1-2s})$ and 
\begin{equation} \label{weak formulation} 
c_{n,s}\langle h, tr_{\Omega}(\zeta)\rangle_{H',H}
 = \int_\C y^{1-2s}\nabla v \nabla \zeta\;dxdy \qquad \hbox{for all } \zeta \in H_{0,L}^1(y^{1-2s}).
\end{equation}
The constant $c_{n,s}$ is the same constant appearing in \eqref{formula g p}. 
\end{remark}

\proof
The case $s=1/2$ was treated in \cite{cabre-tan}. 

The space $H'$ can be identified with the space of distributions $h = \sum_{k=1}^\infty h_k \varphi_k$ such that $\sum_{k=1}^\infty h_k^2 \mu_k^{-s} < \infty$.
Then, it is straight forward to verify that  for any $h \in H'$ there is a unique $u \in H$ such that $(-\Delta)^s u = h$. Fix now $h=\varphi_{k}$ for some $k\ge1$ and let $u=\mu_{k}^{-s}\varphi_{k}$, so that $(-\Delta)^s u=h$. 
By the Lax-Milgram theorem, there is a unique  $v \in H_{0,L}^1(y^{1-2s})$ such that
\eqref{weak formulation} holds. 
Letting $g_{k}$ denote the unique solution of \eqref{bessel ode}--\eqref{bessel ode 2},
by a direct computation, we find that 
$$
v (x,y)=  \mu_{k}^{-s}\varphi_{k}(x)g_{k}(y)
$$
solves \eqref{weak formulation}, with $h=\varphi_{k}$ and its trace is given by $\mu_{k}^{-s}\varphi_{k}=u$. This proves the lemma in the case $h=\varphi_{k}$. By linearity and density, the same holds true for any $h\in H'$.
\hfill\qed

\subsection{Maximum principles}

\begin{lemma} \label{mp}
Let $n\ge 1$ and $\Omega\subset\R^n$ any bounded open set. Take $h\in H'$ and let $u\in H$ de the corresponding solution of \eqref{linear problem}. Let also $v\in H^1_{L}(y^{1-2s})$ denote the canonical extension of $u$. 

If $h\ge0$ a.e. in $\Omega$, then $u\ge0$ a.e. in $\Omega$ and $v\ge0$ in $\C$.  
\end{lemma}

\proof
Simply use $v^-$ as a test function in \eqref{weak formulation}. 
\hfill\qed

\begin{lemma} \label{smp} Let $\Omega\subset\R^n$ denote any domain and take $R>0$. Let $v$ denote any locally integrable function on $\Omega\times(0,R)$ such that
$$
\int_{\Omega\times(0,R)}y^{1-2s}\vert\nabla v\vert^2\;dxdy<+\infty.
$$  
Assume in addition that 
$$
-\nabla\cdot(y^{1-2s}\nabla v)=0\qquad\text{in $\Omega\times(0,R)$},
$$
$v\ge0$ in $\Omega\times(0,R)$, and $\left.-y^{1-2s}v_{y}\right\vert_{y=0}\ge0$ in $\Omega$ in the sense that
$$
 \int_{\Omega\times(0,R)}
y^{1-2s}\nabla v\cdot\nabla\zeta\;dxdy\ge0
$$
for all $\zeta\in H^1(y^{1-2s},\Omega\times(0,R))$ such that $\zeta\ge0$ a.e. in $\Omega\times(0,R)$ and $\zeta=0$ on $\partial \Omega\times(0,R)\,\cup\, \Omega\times\{0\}$.

Then, either $v\equiv0$, or for any compact subset $K$ of $\Omega\times[0,R)$,
$$
\text{ess inf }\; v\vert_{K}>0.
$$
\end{lemma}

\proof
Let $\tilde v$ denote the even extension of $v$ with respect to the $y$ variable, defined in $\Omega\times(-R,R)$ by
$$
\tilde v(x,y)=
\begin{cases}
v(x,y)&\quad\text{if $y>0$,}\\
v(x,-y)&\quad\text{if $y<0$.}
\end{cases}
$$
Then, 
$$
\int_{\Omega\times(-R,R)}y^{1-2s}\nabla \tilde v\nabla \zeta\;dxdy\ge 0,
$$
for all $\zeta\in H^1(y^{1-2s},\Omega\times(-R,R))$, such that $\zeta\ge0$ a.e. in $\Omega\times(-R,R)$ and $\zeta=0$ on $\partial(\Omega\times(-R,R))$. By the results of Fabes, Kenig, and Serapioni (see Theorem 2.3.1 and the second line of equation (2.3.7) in \cite{FKS}), either $\tilde v\equiv0$, or $\text{ess inf }\; \tilde v\vert_{K}>0$ for any compact set $K$ of $\Omega\times(-R,R)$.
\hfill\qed

\begin{lemma} \label{bpl}
Let $\Omega\subset\R^n$ denote an open set satisfying an interior sphere condition at some point $x_{0}\in\partial\Omega$. Let $R>0$ and let $v$ denote any measurable function on $\Omega\times(0,R)$, $v\ge 0$, $v\not\equiv 0$, such that
$$
\int_{\Omega\times(0,R)}y^{1-2s}\vert\nabla v\vert^2\;dxdy<+\infty.
$$  
Assume in addition that 
$$
-\nabla\cdot(y^{1-2s}\nabla v)=0\qquad\text{in}\,\, \Omega\times(0,R),
$$
and $\left.-y^{1-2s}v_{y}\right\vert_{y=0}\ge0$ in $\Omega$ in the sense that
$$
 \int_{\Omega\times(0,R)}
y^{1-2s}\nabla v\cdot\nabla\zeta\;dxdy\ge0
$$
for all $\zeta\in H^1(y^{1-2s},\Omega\times(0,R))$ such that $\zeta\ge0$ a.e. in $\Omega\times(0,R)$ and $\zeta=0$ on $\partial \Omega\times(0,R)\,\cup\, \Omega\times\{0\}$.

Then, there exists $\epsilon>0$ and 
a constant $c=c(R)>0$ such that 
$$
v(x,y)\ge c\vert x-x_{0}\vert\qquad\text{for $x\in B_{\epsilon}(x_{0})\cap\Omega$ and $y\in[0,R-2).$}
$$ 
\end{lemma} 

\proof
Take an interior sphere $B$ which is tangent to $\partial\Omega$ at $x_{0}$. Translating and dilating $\Omega$ if necessary, we may always assume that $B$ is the unit ball centered at the origin. Take $\alpha>n-2$ to be fixed later and consider $z=z(x,y)$ the function defined by
$$
z(x,y) = (1+y^{2s})(e^{-y^2}-e^{-(R-1)^2})(\vert x\vert^{-\alpha}-1)\quad\text{for $x\neq0$ and $y\in[0,R-1]$.}
$$
We compute 
$$
\Delta_{x}z = (1+y^{2s})(e^{-y^2}-e^{-(R-1)^2}) \alpha(\alpha- (N-2)) \vert x\vert^{-\alpha-2},
$$

$$
\lim_{y\to0^+}(-y^{1-2s}z_{y} )= -2s(1-e^{-(R-1)^2})(\vert x\vert^{-\alpha}-1) \qquad\text{for $x\neq0$,}
$$

$$
z_{yy}+\frac{1-2s}y z_{y} = -4e^{-y^2}\left[(1-s) + (1+s)y^{2s} - y^2 - y^{2s+2}\right] (\vert x\vert^{-\alpha}-1).
$$
If $y^2\ge (1+s)$, then $z_{yy}+\frac{1-2s}y z_{y}\ge 0$ and
$
\nabla\cdot (y^{1-2s}\nabla z) \ge 0.
$
If $y^2<(1+s)$, then $z_{yy}+\frac{1-2s}y z_{y} \le C(\vert x\vert^{-\alpha}-1)$. 
Choosing $\alpha$  large enough, we deduce that
$$
\nabla\cdot (y^{1-2s}\nabla z)\ge 0\qquad\text{for all $x\neq0$, $y\in[0,R-1]$. }
$$
Now, let  $v$ be as in the statement of the lemma. By Lemma \ref{smp}, $\text{ess inf }\; v\vert_{K}>0$,  on $K=\partial B_{1/2}\times [0,R-1]$. Choose $\delta>0$ so small that 
$v \ge \delta z$ a.e. on $K$. By the maximum principle, applied in the region $(B_{1}\setminus B_{1/2})\times(0,R-1)$, we deduce that $v\ge\delta z$ in this region.
\hfill\qed

\begin{lemma}
\label{maxp2}
Let $\Omega\subset\R^n$ be a bounded open set with smooth boundary.
Let $v$ denote a measurable function on $\Omega\times(0,+\infty)$, such that
$$
\int_{\Omega\times(0,R)}y^{1-2s}\vert\nabla v\vert^2\;dxdy<+\infty 
\quad\text{for all }R>0.
$$  
Assume that $v\ge 0 $ on $\partial \Omega \times(0,+\infty)$, 
that
$$
-\nabla\cdot(y^{1-2s}\nabla v) \ge 0\qquad\text{in $\Omega\times(0,R)$},
$$
and $\left.-y^{1-2s}v_{y}\right\vert_{y=0}\ge0$ in $\Omega$ in the sense that
$$
 \int_{\Omega\times(0,R)}
y^{1-2s}\nabla v\cdot\nabla\zeta\;dxdy\ge0
$$
for all $\zeta\in H^1(y^{1-2s},\Omega\times(0,+\infty))$ with compact support in $\overline \Omega \times [0,+\infty)$
such that $\zeta\ge0$ and $\zeta=0$ on $\partial \Omega\times(0,R)\,\cup\, \Omega\times\{0\}$. If there exist $C>0$ and $m>0$ such that
\begin{equation}
\label{growth}
|v(x,y)|\le C(1+|y|^m)
\quad\text{for all } (x,y) \in \Omega \times (0,+\infty),
\end{equation}
then $v\ge 0 $ in $\Omega \times (0,+\infty)$.
\end{lemma}
\proof
Take $R>0$ such that $\Omega\subseteq B_R(0)$. Let $\varphi_R$ denote the first eigenfunction of $-\Delta$ in $B_R(0)$ with zero Dirichlet boundary condition and let $\mu_R>0$ be its corresponding eigenvalue.
Let $\lambda>0$ to be chosen and set
$$
z(x,y) = \varphi_R(x) (e^{\lambda y} - \lambda y) .
$$
We compute
$$
\nabla \cdot(y^{1-2s} \nabla z ) = y^{1-2s}
\left[
-\mu_R + \lambda^2 +\lambda^2(1-2s) e^{-\lambda y } \frac{e^{\lambda y }-1}{\lambda y}
\right]
\varphi_R(x) e^{\lambda y} .
$$
By choosing $\lambda>0$ small we have $\nabla \cdot (y^{1-2s} \nabla z )<0$ in $B_R(0) \times(0,+\infty)$.
Let $\epsilon>0$. By \eqref{growth} there exists $L>0$ such that 
$v + \epsilon z \ge 0$ for $x\in \Omega$ and $y \ge L$.
Using the maximum principle in the form of Lemma~\ref{mp} we deduce that $v + \epsilon z \ge 0$ in $\Omega \times (0,L)$. Letting $L \to \infty$ we conclude that  $v + \epsilon z \ge 0$ in $\Omega \times (0,+\infty)$. Finally,  by letting $\epsilon \to 0$ we obtain the stated result.
\qed

\subsection{Interior regularity}

In this section, we study the extension problem \eqref{extension}, when $h$ is bounded or belongs to a H\"older space. The proof of the next lemma can be found in \cite{Cabre-Sire}, Lemma 4.4. 
\begin{lemma}\label{lema_reg}
. 
Let $h\in H'$ and $v\in H^1_{0,L}(y^{1-2s})$ denote the solution of \eqref{extension}.  
Then, for any $\omega\subset\subset\Omega$, $R>0$, we have\\
\\
(i)   If $h\in L^\infty(\Omega)$, then $v\in C^{\beta}(\omega\times [0,R])$, for any  $\beta\in(0,\min(1,2s))$,\\
\\
(ii)  If $h\in C^\beta(\Omega)$ then 
\begin{enumerate}
\item $v\in C^{\beta+2s}(\omega\times [0,R])\quad\text{if $\beta+2s<1$,}$\\
\item $\displaystyle\frac{\partial v}{\partial x_{i}}\in C^{\beta+2s-1}(\omega\times [0,R])\quad\text{if $1<\beta+2s<2$, $i=1,\dots,n$},$\\
\item $\displaystyle\frac{\partial^2 v}{\partial x_{i}\partial x_{j}}\in C^{\beta+2s-2}(\omega\times [0,R])\quad\text{if $2<\beta+2s$, $i,j=1,\dots,n$.}$\\
\end{enumerate}
\end{lemma}

\subsection{Boundary regularity}

\begin{lemma} 
\label{boundary regularity}
Let  $u \in H$ be the solution of
\begin{align}
\label{problem h}
\left\{
\begin{aligned}
(-\Delta )^s u & = h && \mbox{in } \Omega
\\
u & = 0 && \mbox{on } \partial  \Omega
\end{aligned}
\right.
\end{align}
where $h \in L^\infty(\Omega)$. Then
$u\in C^\alpha(\overline{\Omega})$ for all $\alpha\in(0,\min(2s,1))$. 
\end{lemma}

We begin with the following estimate.
\begin{lemma}\label{lemma boundary} 
Let  $u \in H$ be the solution of \eqref{problem h}, 
where $h \in L^\infty(\Omega)$. Then there is constant $C$ such that 
$$
\text{if \ $0<s<1/2$,}\quad
|u(x)| \le C dist(x,\partial \Omega)^{2 s} \| h \|_{L^\infty(\Omega)} \quad \mbox{for all } x \in \Omega,
$$
and 
$$
\text{if $1/2 \le s <1$,}\quad
|u(x)| \le C dist(x,\partial \Omega) \| h \|_{L^\infty(\Omega)} \quad \mbox{for all } x \in \Omega .
$$
\end{lemma}
\proof 
We use a suitable barrier to prove the estimate. To construct it, we 
write $x=(x_1,\dots,x_n)$ and define
$$
\bar h(x) =
\begin{cases}
1 & \hbox{if } x \in B_2, \ x_1 < 0
\\
-1 & \hbox{if } x \in B_2, \ x_1 > 0
\\
0 & \hbox{if } x \not\in B_2 .
\end{cases}
$$
We construct a solution $\bar v$ of the problem
\begin{equation*}
\left \{
\begin{aligned}
{\rm div}\, (y^{1-2s} \nabla v)&=0 &&
{\mbox{ in $\R^n \times (0,+\infty)$}}
\\
v(z)& \to 0
&& \mbox{ as $|z|\to \infty$}
\\
-y^{1-2s} v_y& = \bar h(x)
&& \mbox{ on $\R^n \times \{ 0 \} $}
\end{aligned}\right . \end{equation*}
as
\begin{align}
\label{def bar v}
\bar v (x,y)=
C_{n,s} \int_y^\infty  t \int_{\R^n} 
\frac{\bar h(\tilde x)}{(t^2 + |x-\tilde x|^2)^{\frac{n+2-2 s}{s}}} \, d \tilde x \, d t.
\end{align}
This implies 
$$
\bar v(x,0) =    C_{n,s}' \int_{\R^n}
\frac{ \bar h(\tilde x)}{ |x-\tilde x|^{n-2 s}} \, d \tilde x
\quad x \in \R^n,
$$
where $C_{n,s}' = \frac{C_{n,s}}{n-2s}$.
By our choice of $\bar h$ we can write for $x \in \R^n$
$$
\bar v(x,0) =  - C_{n,s}' ( I(x) - I(-x) )
$$
where
$$
I(x) = \int_{B_2^+}
\frac{1}{|x-\tilde x|^{n-2s}}
\,  d \tilde x
$$
and $B_2^+ = \{ (x_1,\dots,x_n) \in B_2(0) : x_1>0\}$.
From this formula we see that
if $0<s<1/2$ then
$$
|I(x) - I(0)|\le C |x|^{2s}  \quad \mbox{for all } x \in \R^n ,
$$
and
if $1/2 \le s < 1$ then
$$
|I(x) - I(0)|\le C |x| \quad \mbox{for all } x \in \R^n .
$$
These estimates imply that if $0<s<1/2$
\begin{align}
\label{est v1}
|\bar v(x)| \le C |x|^{2 s}  \quad \mbox{for all } x \in \R^n,
\end{align}
and if $1/2 \le s <1$
\begin{align}
\label{est v2}
|\bar v(x)| \le C |x|  \quad \mbox{for all } x \in \R^n.
\end{align}

Now let  $u \in H$ be the solution to \eqref{problem h} with $h \in L^\infty(\Omega)$ and let $v$ denote its canonical extension.
Take a point $x_0 \in \partial \Omega$.
By the smoothness of $\partial \Omega$ we can find $x_1 \in \R^n \setminus \Omega$ and $R>0$ such that $B_R(x_1) \subseteq \R^n \setminus \Omega$ and $x_0 \in \partial B_R(x_1)$. 
We can choose $R$ bounded and bounded below.
By suitable translation and rescaling, we can assume that $x_1=0$, $R=1$ and  $|x_0|=1$.  
After a further rotation we can also assume $x_0 = (1,0,\ldots,0) \in \R^n$.

We will then define a comparison function $w$ as the Kelvin transform of a translate of $\bar v$ as defined by \eqref{def bar v}.  Let $ \tilde v(x,y) = \bar v(x-x_0,y)$.
We write points in $(x,y) \in \R^n \times \R$ as $X = (x,y)$ and $|X|^2=|x|^2+y^2$. We also write $\R^{n+1}_+$ for the set of points $X=(x,y) \in \R^n\times\R$ with $y>0$.
Let 
$$
w(X) = |X|^{2s-n}\tilde v\left(\frac{X}{|X|^2}\right) 
\quad X \in \R^{n+1}_+, X\not=0.
$$
A direct calculation shows that 
$$
{\rm div}\, (y^{1-2s} \nabla w)=0 
\quad\text{in } \R^{n+1}_+
$$
and
$$
\lim_{y\to0^+} ( - y^{1-2s} w_y (x,y) ) = |x|^{-2s-n}
\bar h\left( \frac{x}{|x|^2}\right) \quad
\text{for all } x \in \R^n, x\not=0.
$$
In $\R^n\setminus B_1(0)$ by construction we have  $\bar h(x/|x|^2)=1$. 
Since $\Omega$ is bounded and contained in $\R^n\setminus B_1(0)$,
we see that there is some constant $c>0$ (bounded uniformly from below 
with respect to the parameters $x_0,x_1,R$ with $R$ bounded from 
below) such that 
$$
\lim_{y\to0^+} ( - y^{1-2s} w_y (x,y) ) \ge c
\quad\text{for all } x\in \Omega.
$$
Since $\tilde v>0$ in $B_1(0) \times (0,+\infty)$ we have $w>0$ in $\Omega \times (0,+\infty)$. Then, there is a constant $c>0$ (uniformly bounded from below as $x_0$, $x_1$ and $R$ vary) such that $w(x,1) \ge c$ for all $x\in\Omega$. Since $w \ge 0 $ on $\partial \Omega \times (0,+\infty)$ and $v$ vanishes there, by the maximum principle we have
$$
v \le C \|h\|_{L^\infty(\Omega)} w \quad\text{in } \Omega \times (0,1)
$$
for some $C>0$. From this, \eqref{est v1} and \eqref{est v2} we deduce the stated estimates.
\qed


\proof[Proof of Lemma~\ref{boundary regularity}.]
We use a standard scaling argument combined with 
interior regularity estimates from Lemma~\ref{lema_reg} 
and Lemma~\ref{lemma boundary}. Let $v$ denote the 
canonical extension of $u$ and let us concentrate on the 
case $0<s<1/2$.

Take $x_0,y_0 \in \Omega$.
If $x_0,y_0 $ and satisfy $|x_0 - y_0| \ge dist(x_0,\partial\Omega)/2$ and $|x_0 - y_0| \ge dist(y_0,\partial\Omega)/2$ from  Lemma~\ref{lemma boundary} 
$$
|v(x_0,0) - v(y_0,0)|\le
|v(x_0,0)| +|v(y_0,0)|
\le
C  \|h\|_{L^\infty(\Omega)}|x_0 - y_0|^{2s}
$$
$$
\le
C  \|h\|_{L^\infty(\Omega)}|x_0 - y_0|^{\beta}.
$$

Now suppose that $|x_0-y_0|\le dist(x_0,\partial\Omega)/2$ and let $r=dist(x_0,\partial\Omega)/2$. Consider the function $\tilde v(x,y) = v(x_0 + r x ,r y)$ defined for $x \in B(0,1)$ and $y>0$. Thus 
$$
{\rm div} (y^{1-2s}\nabla \tilde v) = 0 \quad 
\text{in } B_1(0) \times(0,+\infty)
$$
and
$$
\lim_{y\to 0^+} (-y^{1-2s} \tilde v(x,y))
=  \tilde h(x)
\quad x\in B_1(0),
$$
where $\tilde h(x) = r^{2s} h(r x)$.
By Lemma~\ref{lemma boundary}  we find
$$
\sup_{B_1(0)} |\tilde v| \le C r^{2s} \|h\|_{L^\infty(\Omega)} .
$$
Let $0<\beta<2s$.
Using the interior estimate (Lemma~\ref{lema_reg})
$$
\| \tilde v \|_{C^\beta(\overline B_{1/2})}
\le C ( \sup_{B_1}|\tilde v| + \sup_{B_1}|\tilde h|)
\le C r^{2s } \|h\|_{L^\infty(\Omega)} 
$$
we deduce 
$$
|v(x_0,0) - v(y_0,0)|\le C \|h\|_{L^\infty(\Omega)} |x_0 - y_0|^\beta r^{2s-\beta}
\le
C \|h\|_{L^\infty(\Omega)} |x_0 - y_0|^\beta .
$$
The proof in the case $1/2\le s <1$ follows analogously.
\hfill\qed

\section{Proof of Proposition~\ref{exist}} \label{existence} 
We begin by adapting Lemma 1 in \cite{bcmr}:
\begin{lemma}\label{lemma1bcmr} 
Let $n\ge1$ and let $\Omega\subset\R^n$ denote a smooth bounded domain. Take $f\in L^1(\Omega,\varphi_{1}dx)$. Then, there exists a unique $u\in L^1(\Omega,\varphi_{1}dx)$ such that 
\begin{equation} \label{linear}
\left\{
\begin{aligned}
(-\Delta)^s u &= f&\qquad\text{in $\Omega$}\\
u&=0&\qquad\text{on $\partial\Omega$,}
\end{aligned}
\right.
\end{equation}  
in the sense that 
\begin{equation} \label{linear weak} 
\int_{\Omega}u\psi\;dx = \int_{\Omega}f(-\Delta)^{-s}\psi\;dx,\qquad\text{for all $\psi\in C^\infty_{c}(\Omega)$.}
\end{equation} 
In addition, letting $\mu_{1}>0$ denote the principal eigenvalue of the Laplace operator with homogeneous Dirichlet boundary condition on $\partial\Omega$, we have
\begin{equation} \label{linear estimate}
\int_{\Omega} \vert u\vert \varphi_{1}\;dx \le \frac1{\mu_{1}}\int_{\Omega}\vert f\vert \varphi_{1}\;dx. 
\end{equation}  
Moreover, if $f\ge0$ a.e., then $u\ge0$ a.e. in $\Omega$.
\end{lemma}

\proof
Take $\psi\in C^\infty_{c}(\Omega)$. Then, there exists a constant $C>0$ such that $\vert \psi\vert \le C \varphi_{1}$. By the maximum principle (Lemma \ref{mp}), it follows that $\varphi=(-\Delta)^{-s}\psi$ satisfies $\vert\varphi\vert\le \frac{C}{\mu_{1}}\varphi_{1}$. In particular, \eqref{linear weak} makes sense for any $\psi\in C^\infty_{c}(\Omega)$. 

Let $f\in L^\infty(\Omega)\subset H'$. Then, equation \eqref{linear} has a unique solution $u\in H$, i.e. for any $\zeta\in H$,
$$
\sum_{k=1}^{+\infty}\mu_{k}^s u_{k}\zeta_{k} = \sum_{k=1}^{+\infty}f_{k}\zeta_{k},
$$
where $u_{k}=\int_{\Omega}u\varphi_{k}\;dx$, and $\zeta_{k}, f_{k}$ are similarly defined. Take now $\zeta=(-\Delta)^{-s}\psi$, $\psi\in C^\infty_{c}(\Omega)$. Then,
$\zeta_{k}=\mu_{k}^{-s}\psi_{k}$ and 
$$
\sum_{k=1}^{+\infty} u_{k}\psi_{k} = \sum_{k=1}^{+\infty}f_{k}\mu_{k}^{-s}\psi_{k},
$$
which is equivalent to \eqref{linear weak}. We prove next that \eqref{linear estimate} holds. To see this, write $f=f^+-f^-$, where $f^+$ is the positive part of $f$ and $f^-$ its negative part. Without loss of generality, we may always assume that $f\ge0$ a.e. Then, by the maximum principle (Lemma \ref{mp}), $u\ge0$ a.e. and using \eqref{linear weak} with $\psi=\varphi_{1}$, we deduce \eqref{linear estimate}.   
The rest of the proof is the same as that of Lemma 1 in \cite{bcmr}, so we skip it.
\hfill\qed

\proof[Proof of Proposition \ref{exist}]
The method of sub and supersolutions can be applied in the context of solutions of \eqref{problem} belonging to $H\cap L^\infty(B)$. Since $\zeta =0$ is always a subsolution, we begin by showing that there exists a positive supersolution of \eqref{problem} for small $\lambda>0$. Take $\zeta_{0}\in H$ to be the solution of $(-\Delta)^s \zeta_{0}=1$. By Lemma \ref{boundary regularity}  , $\zeta_{0}\in C(\overline\Omega)$ and 
$$
(-\Delta)^s \zeta_{0}=1 \ge\lambda f(\zeta_{0}), \quad\text{for $\lambda\le 1/ \| f(\zeta_{0}) \|_{L^\infty(\Omega)}$.}
$$
Hence,
$$
\lambda^* =\sup \{\lambda>0\; :\; \eqref{problem} \text{ has a solution in $H\cap L^\infty(\Omega)$}\} 
$$
is positive and well-defined. Multiplying \eqref{problem} by $\varphi_{1}$ and using that $f$ is superlinear, we easily deduce that $\lambda^*<+\infty$. It is also clear by the method of sub and supersolutions that \eqref{problem} has a minimal (hence stable), positive solution $u_{\lambda}\in H\cap L^\infty(\Omega)$, for all $\lambda\in(0,\lambda^*)$. 
By minimality, $u_{\lambda}$ increases with $\lambda$. We claim that $u^*(x):=\lim_{\lambda\nearrow\lambda^*}u_{\lambda}(x)$ is a weak solution of \eqref{problem} for $\lambda=\lambda^*$. Take $\lambda<\lambda^*$, $u=u_{\lambda}$ and multiply \eqref{problem} by $\varphi_{1}$. Then,
\begin{equation} \label{esti1}
\mu_{1}\int_{\Omega}u\varphi_{1}\;dx = \lambda \int_{\Omega} f(u)\varphi_{1}\;dx.
\end{equation}    
Since $f$ is superlinear, for every $\epsilon>0$ there exists $C_{\epsilon}>0$ such that, for all $t\ge0$
$
f(t)\ge \frac1\epsilon t - C_{\epsilon}.
$
Hence,
$$
\lambda^* C_{\epsilon}\ge \left(\frac\lambda\epsilon-\mu_{1}\right)\int_{\Omega}u\varphi_{1}\;dx.
$$
Choosing $\epsilon=\frac\lambda{2\mu_{1}}$, we obtain that 
$$
\int_{\Omega}u_{\lambda}\varphi_{1}\;dx \le C,
$$
for some constant $C$ independent of $\lambda$. By \eqref{esti1}, we also have
\begin{equation} \label{loneesti} 
\int_{\Omega}f(u_{\lambda})\varphi_{1}\;dx \le C,
\end{equation} 
and, by monotone convergence, we may pass to the limit as $\lambda\to\lambda^*$ in \eqref{j1}. 
\hfill\qed 

\begin{remark}
Observe that for $s\ge1/2$, we have the stronger estimate 
\begin{equation} \label{esti3} 
\| u_{\lambda} \|_{L^1(\Omega)} \le C,
\end{equation} 
as follows from multiplying \eqref{problem} by $\zeta_{0}$  and using Lemma \ref{lemma boundary}, giving the estimate  
\begin{equation} \label{esti2} \zeta_{0}\le C\varphi_{1}. 
\end{equation} 
Note also that \eqref{esti2} fails for $s<1/2$.
Due to radial monotonicity (see Lemma \ref{gnn}), estimate \eqref{esti3} remains however true if $\Omega=B_{1}$ and $s\in(0,1)$ is arbitrary.     
\end{remark}

\section{Radial symmetry}\label{sec gnn}
\begin{lemma} \label{gnn} Let $u\in H\cap L^\infty(B)$ denote a solution of \eqref{problem}. Then, $u$ is radially decreasing, i.e. $u(x)=u(\rho)$ whenever $\vert x\vert=\rho$, $u$ is smooth in $B$, and 
\begin{equation} \label{monotone}
\frac{\partial u}{\partial \rho}<0\qquad\text{in $B\setminus\{0\}$.}
\end{equation}
In addition, the canonical extension $v$ of $u$ is smooth in $\C$, $v(x,y)=v(\rho,y)$, 
and 
\begin{equation} \label{monotone2}
\frac{\partial v}{\partial \rho}<0\qquad\text{in $\C\setminus\{\rho=0\}$.}
\end{equation}     
\end{lemma}

\proof The smoothness of $u$ and $v$ follows from Lemmata \ref{lema_reg} and \ref{boundary regularity}. 
To prove radial symmetry, \eqref{monotone}, and \eqref{monotone2}, 
we apply the moving plane method  (\cite{gnn}). Thus, it suffices to show that 
$$
\frac{\partial v}{\partial x_{1}}<0\qquad\text{in $\{(x,y)\in B_{1}\times [0,+\infty)\; :\; x_{1}>0\}$.}
$$
Now we show the last statement.  Given $\mu\in(0,1]$, let $T_{\mu}=\{(x,y)\in\R^n\times\R^+\; :\; x_{1}=\mu\}$ and $\Sigma_{\mu}=\{(x,y)\in B_{1}\times[0,+\infty)\; :\; x_{1}>\mu\}$. Let also $v_{\mu}(x,y)=v(2\mu-x_{1},x',y)$ for $(x,y)\in\Sigma_{\mu}$ and $w_{\mu}=v_{\mu}-v$. We claim that  $w_{\mu}\ge 0$ in $\Sigma_{\mu}$, for $\mu$ close to $1$. To prove this, observe that $w=w_{\mu}$ solves
\begin{align*}
\label{extension}
\left\{
\begin{aligned}
{\rm div}\, (y^{1-2s} \nabla w) & =0 &&
\mbox{in $\Sigma_{\mu}$}
\\
w&\ge0 && \mbox{on } \partial_L \Sigma_{\mu}
\\
-y^{1-2s}w_{y} - a(x)w&=0 && \mbox{on } \{x\in B_1:x_{1}>\mu\} \times \{0\},
\end{aligned}
\right.
\end{align*} 
where 
\begin{equation} \label{eq:a} 
a(x)=
\begin{cases}
\frac{f(u_{\mu})-f(u)}{u_{\mu}-u}& \quad\text{whenever $u_\mu\neq u$},\\
$0$&\quad\text{ otherwise.}
\end{cases}
\end{equation} 
Now multiply the above equation by $w^-$ and integrate over $\Sigma_\mu$. Then,
$$
\int_{\Sigma_{\mu}}y^{1-2s}\vert\nabla w^-\vert^2\;dxdy = \int_{\{x\in B_1:x_{1}>\mu\} } a(x)(w^-)^2\;dx.
$$
We extend $w^-$ by $0$ outside $\Sigma_{\mu}$, so that $w^-\in H^1(y^{1-2s};\R^n)$. By the trace theorem (Proposition \ref{trace theorem}), there exists a constant $C_{tr}>0$ such that
$$
\| w^-\|_{H^s(\R^n)}^2 \le C_{tr} \int_{\R^n}y^{1-2s}\vert\nabla w^-\vert^2\;dxdy, 
$$
and by the Sobolev imbedding of $H^s(\R^n)$ into $L^p(\R^n)$, with 
\begin{equation} \label{sobolev} 
\frac1{p}=\frac12 -\frac{s}n,
\end{equation} 
we have
$$
\| w^- \|_{L^p(\R^n)}^2\le C_{S} \| w^- \|_{H^s(\R^n)}^2.  
$$
Hence, by Hoelder's inequality 
\begin{align*}
\left(\int_{\R^n} \vert w^-\vert^p\;dx\right)^{2/p}&\le C_{tr}C_{S} \int_{\{x\in B_1:x_{1}>\mu\} } \vert a(x)\vert (w^-)^2\;dx\\
&\le
C_{tr}C_{S}\left(\int_{\{x\in B_1:x_{1}>\mu\} }(w^-)^p\;dx\right)^{2/p} \left(\int_{\{x\in B_1:x_{1}>\mu\} }\vert a\vert^{\frac{p}{p-2}}\;dx\right)^{1-2/p}
\end{align*}
Since $a$ is uniformly bounded, $\int_{\{x\in B_1:x_{1}>\mu\} }\vert a\vert^{\frac p{p-2}}\;dx\to0$, as $\mu\to1^-$. Therefore, for $\mu$ sufficiently close to $1$, we conclude that  $w^-\equiv0$, and the claim.

Consider now
$$
\mu_{0}=\inf\left\{ \mu\in(0,1)\; : \;
w_{\mu}\ge0 \quad\text{in $\Sigma_{\mu}$}
\right\}.
$$
The above argument shows that $\mu_{0}$ is well-defined and $\mu_{0}<1$. We want to prove that $\mu_{0}=0$. Assume by contradiction that $\mu_{0}>0$. By continuity, $w_{\mu_{0}}\ge0$ in $\Sigma_{\mu_{0}}$, and  by the strong maximum principle (Lemma \ref{smp}), $w_{\mu_{0}}>0$ in $\Sigma_{\mu_0}$. Fix now 
$\epsilon>0$ small, $\mu=\mu_{0}-\epsilon$ and choose a compact set $K\subset \{x\in B_{1}\; :x_{1}>\mu_{0}\;\}$ such that 
$$
C_{tr}C_{S} \left(\int_{\{x\in B_1:x_{1}>\mu\}\setminus K }\vert a\vert^{\frac{p}{p-2}}\;dx\right)^{1-2/p}<\frac12.
$$
Taking $\epsilon>0$ smaller if necessary, we can assume that $w_{\mu}>0$ in $K$.
Arguing as before, we can prove that $w_{\mu}^-\equiv 0$ in $\Sigma_{\mu}\setminus K$, and thus $w_{\mu}\ge0$ everywhere in $\Sigma_{\mu}$, contradicting the definition of $\mu_{0}$.  

We have just proved that $w_{\mu}\ge0$ in $\Sigma_{\mu}$ for all $\mu\in (0,1)$, and  by the strong maximum principle (Lemma \ref{smp}) we find that  $w_{\mu}>0$ in $\Sigma_{\mu}$.
Finally,  by the boundary point lemma (Lemma \ref{bpl}), we conclude
$$
2\frac{\partial v_{\mu}}{\partial x_{1}}(\mu,x',y)= -\frac{\partial w_{\mu}}{\partial x_{1}}(\mu,x',y)<0\qquad\text{for all $(\mu,x',y)\in B_{1}\times[0,+\infty)$,}
$$
as desired.
\hfill\qed

\section{Weighted integrability}

We will use the following notation. Given a point $(x,y)\in\C=B_{1}\times (0,+\infty)$, we let $\rho = \vert x\vert$ and $v_{\rho}=\frac{\partial v}{\partial\rho}$ for any $C^1$ function $v$ defined on $\C$, which depends only on $\rho$ and $y$.

In what follows, for $\lambda \in [0,\lambda^*)$, $u_\lambda$ denotes the minimal solution of \eqref{problem} and $v_\lambda$ its canonical extension, which satisfies
\begin{equation}\label{bdy}
\left \{
\begin{aligned}
{\rm div}\, (y^{1-2s} \nabla v)&=0&& 
\mbox{ in $\C$}
\\
v&= 0&&
\mbox{ on $\partial_L \C$}
\\
-y^{1-2s}v_y& =\lambda f(v)&&
\mbox{ on $B_{1}\times\{0\}$.}
\end{aligned}\right . 
\end{equation}
By elliptic regularity (Lemmata \ref{lema_reg} and \ref{boundary regularity}), for $\lambda\in[0,\lambda^*)$, 
$u_{\lambda}\in C^\infty(B_{1})~\cap~C(\overline{B_{1}})$, and $v_\lambda$ is smooth in $\C$. By Lemma~\ref{lema_reg}, we also deduce that
$v_\lambda \in C^{\alpha}(K\times [0,R])$ for every  compact $K  \subset B_1$ and $R>0$. Moreover, any of the derivatives of $v_\lambda$ with respect to the $x$ variables belongs to $C^{\alpha}(K\times [0,R])$ for every  compact set $K  \subset B_1$ and $R>0$.

The main result in this section is the following.

\begin{prop}
\label{key_lemma}
Assume $n\ge 2$. Let $\lambda\in(0,\lambda^*)$, $u=u_{\lambda}$ be the minimal solution of \eqref{problem} and $v$ its canonical extension. Let $\alpha$ satisfy
\begin{equation}\label{alpha}
1\le \alpha < 1+\sqrt{n-1} .
\end{equation}
Then
\begin{equation}\label{key}
\int_{[\rho\le1/2]}
y^{1-2s}  
v_\rho^2
\rho^{-2\alpha} 
d x d y
\le C
\end{equation}
where $C$ is a constant independent of $\lambda$,
and
$[\rho\le1/2]$ denotes the set $\{ (x,y) \in \C : |x|\le 1/2 \}$.
\end{prop}

We collect in the next lemma some basic estimates expressing that $v_\lambda$ and its derivatives  have exponential decay for $y\ge 1$, which is uniform up to $\lambda<\lambda^*$, and that for fixed $\lambda<\lambda^*$, $v_\rho(\rho,y) = O(\rho)$ as $\rho \to 0$, uniformly as $y \to 0$.

\begin{lemma}
\label{basic lemma}

a) 
There are $\gamma>0$, $C>0$ such that 
\begin{align}
\label{ineq v 1}
v_\lambda(x,y) \le C e^{-\gamma y } \varphi_1(x) \quad
\text{for all } y\ge 1, \ x \in B_1 , \ \lambda \in [0,\lambda^*).
\end{align}
Moreover, for any  $k \ge 0$ there is $C_k>0$ such that
\begin{align}
\label{ineq v 2}
|D^k v_\lambda(x,y)| \le C_k e^{-\gamma y } \quad
\text{for all } y\ge 1, \ x \in B_1 , \ \lambda \in [0,\lambda^*).
\end{align}
The constants $\gamma$ and $C$  are independent of $\lambda$.

b) Given $ \lambda \in [0,\lambda^*)$ and $K$ a compact subset of $B_1$ there exists $C>0$ such that
\begin{align}
\label{bound rho 2}
|\partial_{\rho}v_\lambda(|x|,y)|\le C |x| \quad \forall x\in K, \ y \ge 0.
\end{align}
\end{lemma}
\begin{proof}

a) 
Define $w(x,y) = \varphi_1(x) y^{2 s } e^{-\gamma y }$.
A straight forward computation shows that 
$$
\nabla(y^{1-2s} \nabla w)
= 
\varphi_1(x)  e^{-\gamma y }
\left[ 
(\gamma^2-\lambda_1 )y 
-\gamma (1+2s) 
\right]
$$
and
$$
y^{1-2s}w_y \Big|_{y=0}
= \lim_{y\to 0}y^{1-2s}
\varphi_1(x) e^{-\gamma y } ( -\gamma  y^{2 s } + 2 s  y^{2 s -1} ) = 2 s \varphi_1(x) .
$$
Multiplying equation \eqref{bdy} by $w$ and integrating by parts twice gives
$$
 \lambda \int_{B_1} f(u_\lambda) w \, d x  + 
\int_{B_1} y^{1-2s}w_y v_\lambda \, d x 
+ \int_\C  \nabla(y^{1-2s}\nabla w) v_\lambda = 0.
$$
Recalling that $w(x,o)=0$, we find
$$
2 s \int_{B_1} \varphi_1  u_\lambda \, d x = 
\int_\C   v_\lambda \varphi_1(x)  e^{-\gamma y }
\left[ 
(\lambda_1-\gamma^2 )y 
+\gamma (1+2s) 
\right] \, d x d y.
$$
Now, we choose $0<\gamma< \sqrt{\lambda_1}$ and use estimate $\int_{B_1} \varphi_1  u_\lambda \, d x \le C $  derived in \eqref{loneesti}, to find
\begin{equation} \label{weighted exp decay} 
\int_\C   v_\lambda \varphi_1(x)  e^{-\gamma y }
d x d y 
\le C
\end{equation} 
for all $0\le \lambda < \lambda ^*$.

Let $z$ be the solution to
$$
\left\{
\begin{aligned}
-\Delta z & = 1 && \text{in } B_1
\\
z&=0 && \text{on } \partial B_1 .
\end{aligned}
\right.
$$
For $\tau\ge t>0$ define $\varphi(x,y) =  z(x) ( \tau-y) (y-t)$.
We compute
$$
\nabla (y^{1-2s} \nabla \varphi) = y^{1-2s}
\left[
-( \tau-y) (y-t)
+ z(x) 
\Big(
-2 + (1-2s)
\Big(-2+\frac{\tau+t}{y}
\Big)
\Big)
\right]
$$
Assume that $0<t\le \tau \le 3 t/2$.
We find
$$
\nabla (y^{1-2s} \nabla \varphi) 
\le 
-  y^{1-2s} ( \tau-y) (y-t) .
$$
Multiplying \eqref{bdy} by $\varphi $ and integrating over $B_1 \times (t,\tau)$ we obtain
$$
\tau^{1-2s} 
(\tau-t)
\int_{B_1} v_\lambda(x,\tau) z(x)  \, d x 
+ t^{1-2s} 
(\tau-t)
\int_{B_1} v_\lambda(x,t) z(x)  \, d x 
$$
$$
=
- \int_{B_1 \times (t,\tau) } y^{1-2s} v_\lambda \nabla (y^{1-2s} \nabla \varphi) \;dxdy
\ge  
\int_{B_1 \times (t,\tau) } y^{1-2s} v_\lambda ( \tau-y) (y-t) 
\, d x d y .
$$
Thus, for $t\ge 6$ we deduce
\begin{align*}
\int_{B_1 \times (t+1,t+2) } y^{1-2s} v_\lambda \, d x d y
\le & C t^{1-2s}
\int_{B_1} v_\lambda(x,t) z(x) \, d x 
\\
& 
+  C(t+3)^{1+2s}
\int_{B_1} v_\lambda(x,t+3) z(x) \, d x .
\end{align*}
Integrating this inequality with respect to $t \in [6,13]$, recalling that $z\le C\varphi_{1}$ for some $C>0$,
and using  \eqref{weighted exp decay} we obtain
$$
\int_{B_1\times [8,11] } v_\lambda \, d x dy \le C
$$
with a constant independent of $\lambda$ as $\lambda\to\lambda^*$.

This inequality and standard elliptic estimates imply
\begin{align}
\label{ineq y=1}
v_\lambda(x,y) \le C e^{-\gamma y} \varphi_1(x) \quad
\text{for all } y \in [9,10], \  x \in B_1, \ \text{and} \  \lambda \in [0,\lambda^*).
\end{align}
Now let $\bar w(x,y) = C \varphi_1(x) e^{-\gamma y}$. For $0<\gamma < \sqrt{\lambda_1}$, this is  a supersolution of the equation in  \eqref{bdy} and by comparison in $B_1 \times (1,+\infty)$, using \eqref{ineq y=1}, we deduce \eqref{ineq v 1}.
Inequality \eqref{ineq v 2} is a consequence of \eqref{ineq v 1} and elliptic estimates.

b) This part follows from the fact that for $\lambda<\lambda^*$, $u_\lambda$ is smooth  in $B_1$ and hence $v_\lambda$ and its derivatives with respect to the $x$ variables are in $C^{\alpha}(K \times [0,R])$ for any compact $K \subset B_1$ and $R>0$.
\end{proof}

The following result is a version of Lemma~1 of \cite{CC04} in the case
of radially symmetric functions. 

\begin{lemma}\label{lem:bdstab}
Given $\lambda\in(0,\lambda^*)$, let $u=u_{\lambda}\in H\cap L^\infty(B_{1})$ denote the minimal solution of \eqref{problem}, and let $v\in H^1_{0,L}(y^{1-2s})$ denote its canonical extension. Then, for every $\eta\in C^1({B_{1}}\times [0,+\infty))$ with compact support in $\C$, but not necessarily vanishing on $B_1\times\{0\}$, we have
\begin{equation}
\int_{\C}y^{1-2s}v_\rho^2|\nabla\eta|^2\;dxdy
\ge (n-1)\int_{\C}y^{1-2s}\frac{v_\rho^2}{\rho^2}\eta^2\; dxdy.
\label{eq:3-lm}
\end{equation}
\end{lemma}

\begin{proof}
Inequality
\eqref{stability} implies that for all $\xi\in H^1_{0,L}(y^{1-2s})$, there holds
\begin{equation} \label{stability2}
\int_{\C}y^{1-2s}\vert \nabla \xi\vert^2\;dxdy \ge \int_{B_{1}}f'(u)\xi^2\;dx,
\end{equation} 
where in the right-hand-side integral we identified $\xi$ and its the trace.

Let $\eta\in C^1(B_{1}\times [0,+\infty))$ as in the statement of the lemma and 
take $\xi=\eta v_{\rho}$.
By Lemma \ref{basic lemma},   $\xi \in H^1_{0,L}(y^{1-2s})$ and from  \eqref{stability2} we obtain
\begin{align} \label{somelabel} 
\lambda\int_{B_1}f'(u)v_{\rho}^2\eta^2dx 
& \le \int_{\C}y^{1-2s}|\nabla(v_{\rho}\eta)|^2dxdy
\nonumber\\
& \le \int_{\C}y^{1-2s}\{|\nabla v_{\rho}|^2\eta^2+v_{\rho}^2|\nabla\eta|^2+v_{\rho}\nabla
v_{\rho}\cdot\nabla\eta^2\}\;dxdy
 \nonumber\\
& \le \int_{\C}y^{1-2s}\{v_{\rho}^2|\nabla \eta|^2+\nabla(\eta^2 v_{\rho})\cdot\nabla v_{\rho}\}\;dxdy.
\end{align}

Since by Lemma~\ref{gnn}, $u$ is radially symmetric, by differentiation of \eqref{bdy} with respect to $\rho$, one gets
\begin{equation}
 \nabla \cdot (y^{1-2s}\nabla v_\rho)=y^{1-2s}\frac{n-1}{\rho^2}v_\rho\qquad\text{in $\C$.}
\label{eq:3-lm2}
\end{equation}
Next, we differentiate the Neumann boundary condition in \eqref{bdy}  with respect to $\rho$ to obtain
\begin{equation}
-y^{1-2s}\partial_y v_\rho=\lambda f'(v)v_\rho
\qquad \text{for $0 \le \rho <1$}.
\label{diff_neumann}
\end{equation}
Now, we multiply \eqref{eq:3-lm2} by $\eta^2 v_\rho$, and integrate by parts and use \eqref{diff_neumann} to find
$$
\int_{\C}y^{1-2s} \nabla(\eta^2 v_{\rho})\cdot\nabla v_{\rho} \;dxdy =
\lambda\int_{B_{1}}f'(u)(v_{\rho}\eta)^2\;dx -
(n-1)\int_{\C}y^{1-2s}\frac{(v_{\rho}\eta)^2}{\rho^2}\;dxdy.
$$
Combining the last equation with  \eqref{somelabel} yields  \eqref{eq:3-lm}.  
\end{proof}

\medskip

\proof[Proof of Proposition~\ref{key_lemma}]

Given $\ep>0$ let $\zeta_\ep \in C^\infty(\R)$ be such that $\zeta_\ep(t)=0$ for $t\le \ep$ and $t\ge 3/4$, $\zeta_\ep(t)=1$ for $t \in [2\ep,1/2]$, and $\zeta´(t) \le C/\ep$ for $t \in [\ep,2\ep]$.
Given $R>0$ we let $\phi_R$ denote a function  $C^\infty(\R)$ such that $\psi_R(y) =1$ for all $r\le R$ and $\psi_R(y)=0$ for all $y\ge R+1$.

Let $\alpha$ satisfy \eqref{alpha} and for $\ep>0$, $R>0$ define $\eta(\rho,y) = \rho^{1-\alpha} \zeta_\ep(\rho) \psi_R(y)$.
Given $\delta>0$ we estimate
$$
|\nabla \eta|^2 
\le 
((1-\alpha)^2 + \delta )\rho^{-2\alpha} \zeta_\ep(\rho)^2 \psi_R(y)^2 + C_\delta\rho^{2-2\alpha} | \nabla  (\zeta_\ep \psi_R) |^2
$$
for some $C_\delta>0$.
Then by \eqref{eq:3-lm}
$$
(n-1) \int_\C y^{1-2s}  v_\rho^2
\rho^{-2\alpha}(\zeta_\ep \psi_R)^2
d x d y
\le
((1-\alpha)^2 + \delta )
\int_\C y^{1-2s}  v_\rho^2
\rho^{-2\alpha}(\zeta_\ep \psi_R)^2
d x d y
$$
$$
+ C \int_\C
y^{1-2s}
\rho^{2-2\alpha} v_\rho^2
| \nabla  (\zeta_\ep \psi_R) |^2
d x d y.
$$
Choosing $\delta>0$ small enough 
$$
\int_\C y^{1-2s}  v_\rho^2
\rho^{-2\alpha}(\zeta_\ep \psi_R)^2
d x d y
\le 
C
\int_\C
y^{1-2s}
\rho^{2-2\alpha} v_\rho^2 (
| \nabla  \zeta_\ep|^2 \psi_R^2 + \zeta_\ep^2 |\nabla \psi_R |^2
)
d x d y
$$
where $C>0$.
Thanks to \eqref{bound rho 2} we have
$$
\int_\C
y^{1-2s}
\rho^{2-2\alpha} v_\rho^2 
| \nabla  \zeta_\ep|^2 \psi_R^2 
d x d y
\le \frac C{\ep^2}
\int_{
[\ep\le\rho\le 2 \ep, 0\le y \le R+1]}
y^{1-2s}
\rho^{4-2\alpha}
d x d y
$$
$$
\le C (R+1)^{2-2s}\ep^{2-2\alpha+n} .
$$
Because of  \eqref{alpha}we have that $2-2\alpha+n>0$. Letting $\ep\to 0$ we  find
$$
\int_{[\rho\le1/2,y\le R]}
y^{1-2s}  
v_\rho^2
\rho^{-2\alpha} 
d x d y
\le 
C
\int_{[1/2\le \rho \le 3/4] \cup [R\le y\le R+1]}
y^{1-2s}
\rho^{2-2\alpha} v_\rho^2 
d x d y \le C
$$
where the last inequality follows from \eqref{ineq v 2}.
Finally, letting $R\to \infty$ we conclude \eqref{key}.
\qed
\medskip

For  $0<\beta<n$ we define
\begin{align}
\label{def a}
A_{n,s,\beta}  = 
\int_{\R^n  \times (0,+\infty) }   
\frac{y^{3 - 2 s}}
{(|x|^2 + y^2)^{\frac{\beta+2}{2}}
( y^2 + | x - e |^2)^{\frac{n+2-2 s}{2}}}
\, d x d y
\end{align}
where $r=|(x,y)|=(\rho^2+y^2)^{1/2}$, and
$e$ is any unit vector in $\R^n$.

\begin{lemma}
\label{lemma a}
We have   $1 - \beta C_{n,s} A(n,s,\beta) > 0$, 
where $C_{n.s}$ is the constant in the representation formula \eqref{def bar v}.
\end{lemma}
\proof
Let $h \in L^\infty(\Rn)$ be radial and have compact support, and 
 $u(x,y) = u(\rho,y)$ be a solution of
\begin{equation}
\label{eq v Rn}
\left \{
\begin{aligned}
{\rm div}\, (y^{1-2s} \nabla u)&=0 &&
{\mbox{ in $\R^n \times (0,+\infty)$}}
\\
u(x,y)& \to 0
&& \mbox{ as $|(x,y)|\to \infty$}
\\
-y^{1-2s}u_y& = h(x)
&& \mbox{ on $\R^n \times \{ 0 \} $.}
\end{aligned}\right . 
\end{equation}
Now, we claim that, for any $0<\beta<n$
\begin{align}
\label{result}
0 = 
(1 - \beta C_{n,s} A_{n,s,\beta} )
 \int_{\R^n  } h(x) \rho^{-\beta} \, d x
+
\beta \int_{\R^n  \times (0,+\infty) } 
y^{1-2s}  r^{-\beta-2}  \rho u_\rho \, d x d y.
\end{align}
Assuming the claim for a moment we prove the lemma. Choose a smooth radially decreasing function $h \ge0$, $h \not\equiv 0$ with compact support. Let $u$ be the solution of \eqref{eq v Rn}. By \eqref{def bar v}, $u$ can be explicitly given by a convolution kernel. In turn, this shows that $u$ is radial with respect to $x$ and non-increasing in $|x|$. Hence
$$
\int_{\R^n  \times (0,+\infty) } y^{1-2s}  r^{-\beta-2}  \rho u_\rho \, d x d y < 0
$$
and
$$
\int_{\R^n  } h(x) \rho^{-\beta} \, d x > 0.
$$
This shows that $1 - \beta C_{n,s} A_{n,s,\beta} >0$.

Now we give the argument for  \eqref{result}.
Let $\ep>0$, $\beta  \in (0,  n+2-2s)$ and multiply equation \eqref{eq v Rn}
by $(\rho^2  +y^2 + \ep)^{-\beta/2}$ to get
$$
0 = \int_{\R^n  \times (0,+\infty) }
{\rm div}\, (y^{1-2s} \nabla u) (\rho^2  +y^2 + \ep)^{-\beta/2} \, d x d y
$$
$$
= - \int_{\R^n  } y^{1-2s} u_y (\rho^2  + \ep)^{-\beta/2} \, d x
+ \beta \int_{\R^n  \times (0,+\infty) } y^{1-2s} (\rho^2   +y^2 + \ep)^{-\beta/2-1} (  x \cdot \nabla_x u + y\  u_y )  \, d x d y
$$
$$
= \int_{\R^n  } h(x)  (\rho^2  + \ep)^{-\beta/2}  \, d x
+
\beta \int_{\R^n  \times (0,+\infty) } y^{1-2s}  (\rho^2   +y^2 + \ep)^{-\beta/2-1}  \rho u_\rho \, d x d y
$$
$$
+
\beta \int_{\R^n  \times (0,+\infty) } y^{2-2s}  (\rho^2   +y^2 + \ep)^{-\beta/2-1}   u_y \, d x d y
$$
Using the representation formula
$$
- y^{1-2 s} u_y(x,y) =  C_{n,s} y^{2-2s} \int_{\R^n} \frac{h(\tilde x)}{(y^2 + |x-\tilde x|^2)^{\frac{n+2-2 s}{2}}} \, d \tilde x
$$
we find
$$
0 = \int_{\R^n  } h(x) (\rho^2+\ep)^{-\beta/2} \, d x
+
\beta \int_{\R^n  \times (0,+\infty) } y^{1-2s}   (\rho^2   +y^2 + \ep)^{-\beta/2-1}  \rho u_\rho \, d x d y
$$
$$
-
\beta C_{n,s} \int_{\R^n  \times (0,+\infty) } \int_{\R^n}  y^{3-2s}   (\rho^2   +y^2 + \ep)^{-\beta/2-1} 
\frac{h(\tilde x)}{(y^2 + |x-\tilde x|^2)^{\frac{n+2-2 s}{2}}} \, d \tilde x
\, d x d y.
$$
By Fubini,  the last integral becomes
$$
\int_{\R^n} h(\tilde x) \int_{\R^n  \times (0,+\infty) }  
\frac{ y^{3 - 2 s}}{ (|x|^2   +y^2 + \ep)^{\beta/2+1}  (y^2 + |x-\tilde x|^2)^{\frac{n+2-2 s}{2}}}
\, d x d y \, d \tilde x,
$$
and by the change variables:  $ y =  |\tilde x|  y'$, $y>0$, $x =  |\tilde x|  x'$, $x' \in \R^n$,
we find
$$
 \int_{\R^n  \times (0,+\infty) }  
\frac{ y^{3 - 2 s}}{ (|x|^2   +y^2 + \ep)^{\beta/2+1}  (y^2 + |x-\tilde x|^2)^{\frac{n+2-2 s}{2}}}
\, d x d y
= |\tilde x|^{  - \beta } 
A_{n,s,\beta}
(\frac\ep{|\tilde x|^2})
$$
where
$$
A_{n,s,\beta}
(t)
= 
\int_{\R^n  \times (0,+\infty) }   
\frac{y^{3 - 2 s}}
{(|x|^2 + y^2 + t)^{\frac{\beta+2}{2}}
( y^2 + | x - \frac{\tilde x}{|\tilde x|} |^2)^{\frac{n+2-2 s}{2}}}
\, d x d y.
$$
Therefore, from the above computations we get
\begin{align}
\label{int 1}
\nonumber
0 =  & \int_{\R^n  } h(x) (\rho^2+\ep)^{-\beta/2} (1 -
\beta C_{n,s} A_{n,s,\beta}(\ep/|\tilde x|^2 ) ) \, d x
\\
& +
\beta \int_{\R^n  \times (0,+\infty) } y^{1-2s}  (\rho^2+y^2+\ep)^{-\beta/2-1}  \rho u_\rho \, d x d y
\end{align}
Notice that 
$$
\lim_{\ep \to 0} A_{n,s,\beta}(\ep/|\tilde x|^2 )
= A_{n,s,\beta} 
\quad  \text{for all } \tilde x \in \R^n
$$
and that this limit  is finite for $0<\beta<n+2-2s$.
Moreover $A_{n,s,\beta}$ is independent of $\tilde x$.
Since $\beta<n$ and $h$ is bounded with compact support the function $h(\rho) \rho^{-\beta}$ is integrable. 
Hence, by  letting $\ep\to0$ in \eqref{int 1} we obtain \eqref{result}.
\qed

\section{Proof of Theorem~\ref{thm_extremal}}

\begin{lemma}
Let $h \in L^\infty( B_1)$  and $u \in H$ be the unique solution of
$$
(-\Delta )^{s}u = h \quad \text{in } B_1.
$$
Then
\begin{equation}\label{bnd_urho}
|u(x)|
\ \leq\ C_{n,s}
\int_{B_1}\frac{|h(\tilde x)|}{|x-\tilde{x}|^{n-2 s}} \ d\tilde{x} 
\qquad\text{for every } \ x\in B_1 .
\end{equation}
\end{lemma}
\proof
Writing $h = h^+ - h^- $ with $h^+,h^- \ge 0$ we see that it is sufficient to prove the result in the case $h\ge 0$, so that also $u \ge 0$.

Let $v$ be the canonical extension of $u$.
Since  $v(x,\infty)=0$, for every $x$, we can write
\begin{align}\label{v_0}
v(x,0)=-\int_0^\infty v_y(x,y)\ d y
\quad \text{for all }x\in B_1.
\end{align}
Let $g(x)$ be equal to $ h(x)$ extended by $0$
in $\R^n\setminus B_1$, and denote by $\tv$ the solution of
\begin{equation}
\label{linear_full_2}
\left \{
\begin{aligned}
{\rm div}\, (y^{1-2s} \nabla \tv )&=0 &&
\mbox{ in $\R^n \times (0,+\infty)$}
\\
\tv(z)& \to 0
&& \mbox{ as $|z|\to \infty$}
\\
-y^{1-2s}\tv_y& = g(x)
&& \mbox{ on $\R^n \times \{ 0 \} $.}
\end{aligned}\right . 
\end{equation}
By the Green's representation formula for \eqref{linear_full_2}, we have
\begin{equation}\label{grf}
-\tv_y(x,y)= C_{n,s}  \  y\int_{\R^n }\ \frac{g(\tilde{x})}
{(|x-\tilde{x}|^2+y^2)^{\frac{n+2-2 s}{2}}}\ d\tilde{x}.
\end{equation}
Consider the functions $w = - y^{1-2s} v_y$ and $\tilde w =  - y^{1-2s} \tilde v_y$. 
Then, $w$ and $\tilde w$ satisfy
$$
\nabla (y^{2s-1} \nabla w) = 0 \quad \text{in } \C.
$$
Since $- \tv_y \ge 0  $ in $\R^n \times [0,+\infty)$ in particular we have
$$
\tilde w \ge 0= w\quad\text{on }\partial_L\C.
$$
Furthermore
$$
w \le \tilde w \quad \text{in }  B_1 \times \{ 0 \}
$$
and for $z\in \C$, $w(z),\tilde w(z) \to 0$ as $|z|\to+\infty$. 
Then, the maximum principle (Lemma~\ref{maxp2}) implies that
\begin{equation}\label{ppimax}
-v_y\le -\tv_y\quad\text{in }\C.
\end{equation}
Combining \eqref{v_0}, \eqref{ppimax} together with \eqref{grf} we find
\begin{align*}
v(x,0)&\le C_{n,s} \ \int_0^\infty y\int_{\R^n}\ \frac{g(\tilde{x})}{(|x-\tilde{x}|^2+y^2)^{\frac{n+2-2 s}{2}}}\ d\tilde{x} dy
\\
&= C_{n,s}
 \int_{\R^n} g(\tilde{x})\left(\int_0^\infty  \frac{y}{(|x-\tilde{x}|^2+y^2)^{\frac{n+2-2 s}{2}}}\ dy \right)d\tilde{x}
\end{align*}
for all $x \in B_1$, 
where we have used Fubini's theorem in the last  line.
Claim  \eqref{bnd_urho} follows by performing the integration over the $y$ variable in the last expression,
and recalling the definition of $g(x)$.
\qed

\proof[Proof of Theorem~\ref{thm_extremal}]
We denote points in $\C= B_1 \times (0,+\infty)$ as $(x,y) \in \C$, 
where $x\in B_1$, $y\in (0,+\infty)$, and $\rho = |x|$.

\medskip
\noindent
{\bf Step 1.}
Take $\alpha$ such that  \eqref{alpha} holds.
We claim that for $\beta>0$ such that $2 ( \beta+s- \alpha)<n$ 
we have
\begin{align}
\label{weighted estimate f}
\int_{B_1} f(u_\lambda) \rho^{-\beta} \, d x \le C
\end{align}
with $C$ independent of $\lambda$ as $\lambda\to\lambda^*$.

To prove the claim, let $\ep>0$,  $R>0$ and multiply \eqref{extension} by 
$(\rho^2+y^2+\ep)^{-\beta/2}$ and integrate over $[\rho \le 1/2 , 0\le y \le R]$
to get
$$
0 = \int_{[\rho \le 1/2 , 0\le y \le R]}
\nabla \cdot (y^{1-2s} \nabla v ) (\rho^2+y^2+\ep)^{-\beta/2} \, d x d y.
$$
Integrating by parts we find
\begin{align}
\label{eq weighted f}
\lambda \int_{B_{1/2}} f(u_\lambda) (\rho^2+\ep)^{-\beta/2} \, d x 
= - I_1 -  I_2 + I_3
\end{align}
where
\begin{align*}
I_1 & = \int_{[ \rho \le 1/2 ]}
R^{1-2s} v_y(\rho,R) (\rho^2 + R^2 + \ep)^{-\beta/2} \, d x 
\\
I_2 & = \int_{ [0\le y \le R]}
y^{1-2s} v_\rho(1/2,y)  (  1/4 + y^2 + \ep)^{-\beta/2}\, d y
\\
I_3 & = -\beta
\int_{[ \rho \le 1/2 , 0\le y \le R]}
y^{1-2s}   (\rho^2+y^2+\ep)^{-\beta/2-1} ( v_\rho \rho + v_y y ) \, d x d y.
\end{align*}
By \eqref{ineq v 1} and \eqref{ineq v 2}, $I_1 $ and $I_2$ remain uniformly bounded as $\ep\to0$ and  $\lambda \to \lambda^*$. We decompose further
$$
I_3 = I_\rho + I_y
$$
where
\begin{align*}
I_\rho & = 
-\beta
\int_{[ \rho \le 1/2 , 0\le y \le R]}
y^{1-2s}   (\rho^2+y^2+\ep)^{-\beta/2-1}  v_\rho \rho  \, d x d y,
\\
I_y & = -\beta
\int_{[ \rho \le 1/2 , 0\le y \le R]}
y^{1-2s}   (\rho^2+y^2+\ep)^{-\beta/2-1}  v_y y  \, d x d y.
\end{align*}
Now we  estimate $I_y$. 
Let $g(x)$ be equal to $ \lambda f(u_\lambda(x))$ extended by $0$
in $\R^n\setminus B_1$, and denote by $\tv$ the solution of
\begin{equation}
\label{linear_full_2 b}
\left \{
\begin{aligned}
{\rm div}\, (y^{1-2s} \nabla \tv )&=0 &&
\mbox{ in $\R^n \times (0,+\infty)$}
\\
\tv(z)& \to 0
&& \mbox{ as $|z|\to \infty$}
\\
-y^{1-2s}\tv_y& = g(x)
&& \mbox{ on $\R^n \times \{ 0 \} $.}
\end{aligned}\right . 
\end{equation}
By the Green representation formula for \eqref{linear_full_2 b}, we have
\begin{equation}\label{grf b}
-\tv_y(x,y)= C_{n,s}  \  y
\int_{\R^n}\ \frac{g(\tilde{x})}
{(|x-\tilde{x}|^2+y^2)^{\frac{n+2-2 s}{2}}}\ d\tilde{x}.
\end{equation}
Consider the functions $w = - y^{1-2s} v_y$ and $\tilde w =  - y^{1-2s} \tilde v_y$. 
Then, $w$ and $\tilde w$ satisfy
$$
\nabla\cdot (y^{2s-1} \nabla w) = 0 \quad \text{in } \C.
$$
Since $- \tv_y \ge 0  $ in $\R^n \times [0,+\infty)$ we have in particular
$$
\tilde w \ge 0= w\quad\text{on }\partial_L\C.
$$
Furthermore
$$
w \le \tilde w \quad \text{in }  B_1 \times \{ 0 \}
$$
and for  $z\in \C$, $w(z),\tilde w(z) \to 0$ as $|z|\to+\infty$. 
Then, the maximum principle (Lemma~\ref{maxp2}) implies that
$$
-v_y\le -\tv_y\quad\text{in }\C.
$$
It follows that
$$
I_y 
\le - \beta
\int_{[ \rho \le 1/2 , 0\le y \le R]}
y^{1-2s}   (\rho^2+y^2+\ep)^{-\beta/2-1}  \tilde v_y y  \, d x d y
$$
and by \eqref{grf b}
$$
I_y \le \beta C_{n,s} 
\int_{[ \rho \le 1/2 , 0\le y \le R]}
\int_{\R^n}
\frac{y^{3-2s}   g(\tilde{x})}
{ (\rho^2+y^2+\ep)^{\beta/2+1}  
(|x-\tilde{x}|^2+y^2)^{\frac{n+2-2 s}{2}}}\ d\tilde{x}
\, d x d y
$$
$$
\le \beta C_{n,s} 
\int_{\R^n}
g(\tilde{x})
\left( 
\int_{[ \rho \le 1/2 , 0\le y \le R]}
\frac{y^{3-2s} }
{ (\rho^2+y^2+\ep)^{\beta/2+1}  
(|x-\tilde{x}|^2+y^2)^{\frac{n+2-2 s}{2}}}
\, d x d y
\right)
d \tilde x .
$$
Therefore
\begin{align}
\label{ineq I y}
I_y \le \beta C_{n,s}  A_{n,s,\beta}
\int_{\R^n}
g(\tilde{x}) \, d \tilde x 
=  \beta C_{n,s}  A_{n,s,\beta}
\lambda
\int_{B_1}
f(u_\lambda) \, d x
\end{align}
where $A_{n,s,\beta}$ is defined as in \eqref{def a}.
Combining \eqref{eq weighted f} with \eqref{ineq I y} we obtain
\begin{align}
\label{ineq weighted f}
( 1 - \beta C_{n,s} A_{n,s,\beta})  \lambda 
\int_{B_1}
f(u_\lambda) \, d x \le \lambda \int_{B_1 \setminus B_{1/2}} f(u_\lambda) \, d x -I_1 - I_2 + I_\rho.
\end{align}
Recall that by \eqref{ineq v 1} and \eqref{ineq v 2}, 
\begin{align}
\label{bd I1 I2}
|I_1|\le C , \qquad |I_2|\le C 
\end{align}
for some $C$ independent $\ep\to0$ and  $\lambda \to \lambda^*$.

By the Cauchy-Schwarz inequality
\begin{align*}
|I_\rho | \le & 
\beta
\left(
\int_{[ \rho \le 1/2 , 0\le y \le R]}
y^{1-2s}    v_\rho^2 \rho^{-2\alpha}   \, d x d y
\right)^{1/2}
\left(
\int_{[ \rho \le 1/2 , 0\le y \le R]}
\frac{ y^{1-2s}\rho^{2+2\alpha} }
{(\rho^2+y^2+\ep)^{\beta+2}}  \, d x d y
\right)^{1/2}
\end{align*}
The last integral can be estimated by
$$
\int_{[ \rho \le 1/2 , 0\le y \le R]}
\frac{ y^{1-2s}\rho^{2-2\alpha} }
{(\rho^2+y^2+\ep)^{\beta+2}}  \, d x d y
\le 
\int_0^\infty
\int_{[\rho\le1/2]}
\frac{ y^{1-2s}\rho^{2\alpha} }
{(\rho^2+y^2)^{\beta+1}}  \, d x d y
$$
We change variables $y = \rho t$ for $\rho>0$. Since $\beta >0$,  we have
$$
\int_{[ \rho \le 1/2 , 0\le y \le R]}
\frac{ y^{1-2s}\rho^{2+2\alpha} }
{(\rho^2+y^2+\ep)^{\beta+2}}  \, d x d y
\le 
\int_{[\rho\le1/2]}
\rho^{2\alpha-2\beta - 2s}
\, d x 
\int_0^\infty
\frac{ t^{1-2s}}
{(1+t^2)^{\beta+1}}
 \, d t 
$$
and this integral is finite if $2 ( \beta+s- \alpha)<n$.
The integral  $\int_{[ \rho \le 1/2 , 0\le y \le R]}
y^{1-2s}    v_\rho^2 \rho^{-2\alpha}   \, d x d y
$ remains bounded as $\ep\to0$ and $\lambda\to\lambda^*$ by \eqref{key}, provided $\alpha$ satisfies \eqref{alpha}.

Thus, if $\alpha$ satisfies \eqref{alpha} and  $2 ( \beta+s- \alpha)<n$ we deduce that 
\begin{align}
\label{bd I rho}
|I_\rho|\le C
\end{align}
with $C$ independent of $\ep>0$ and $\lambda \in [0,\lambda^*)$.

By Lemma~\ref{lemma a} we have $ 1 - \beta C_{n,s} A_{n,s,\beta} > 0$. Therefore, from \eqref{ineq weighted f}, \eqref{bd I1 I2} and \eqref{bd I rho}, and using a uniform bound for $u_\lambda$ in $B_1 \setminus B_{1/2}$ we deduce
\eqref{weighted estimate f}.

\medskip
\noindent
{\bf Step 2.} Conclusion.

(a) Assume first that  $n < 2 ( s + 2 + \sqrt{2(s+1)}) $.
Then, $n/2-s< 1+\sqrt{n-1}$ and we can choose $\alpha$ satisfying $n/2-s<\alpha < 1+\sqrt{n-1}$.
Thus, $n-2s < n/2 + \alpha-s$ and we may choose
$\beta=n-2s$ 
in \eqref{weighted estimate f}, which implies that $\int_{B_1} f(u_\lambda) \rho^{-n+2s} \, d x \le C$ with a constant independent of $\lambda$.
By \eqref{bnd_urho} we have
$$
u_\lambda(0)\le C_n\int_{B_1}\rho^{-n+2s}f(u_\lambda(\rho))dx \le C .
$$
Since $u_\lambda$ is radially decreasing, we conclude that $u_\lambda$ is uniformly bounded in $B_1$ as $\lambda\to\lambda^*$.

\medskip
(b) Now assume that $n \ge  2 ( s + 2 + \sqrt{2(s+1)}) $.
Suppose $1<\alpha<1+\sqrt{n-1}$, $\beta>0$, and $2 ( \beta+s- \alpha)<n$. 
Then, using that $f'>0$, that $u_\lambda$ is radially decreasing, as well as the estimate \eqref{weighted estimate f}, we have for $\rho\le 1/2$
$$
c \rho^{n-\beta} f(u_\lambda(\rho)) = f(u_\lambda(\rho)) \int_{B_{2 \rho}\setminus B_\rho} |x|^{-\beta}\, d x \le \int_{B_1} f(u_\lambda) |x|^{-\beta} \, d x \le C
$$
where $c>0$. This yields
$$
f(u_\lambda(\rho)) \le C \rho^{\beta-n} \quad \text{for } 0<\rho \le 1
$$
where $C$ is independent of $\lambda$. Using \eqref{bnd_urho}, this implies that if additionally $\beta < n -2s$, then
$$
u_\lambda(x) \le \frac{C}{|x|^{n-\beta-2s}} 
\qquad \text{for all } x\in B_1.
$$
Since we have the restrictions $\beta<n/2+\alpha-s$ and $\alpha<1+\sqrt{n-1}$, we see that for any $\mu < n/2-s-1-\sqrt{n-1}$, there is $C$ independent of $\lambda$ such that 
$$
u_\lambda(x) \le \frac{C}{|x|^\mu} 
\qquad \text{for all } x\in B_1.
$$
By letting $\lambda\to\lambda^*$ in the last expression we conclude the proof.
\qed

\section{Acknowledgments}
A.C. was partially supported by MTM2008-06349-C03-01 and PAPIIT IN101209.
J. D. was partially supported by   Fondecyt 
1090167 and FONDAP grant  for Applied Mathematics, Chile.
This work is
also part of the MathAmSud NAPDE project (08MATH01)
and ECOS contract no. C09E06 (J.D. \& L.D). Y.S. is supported by the A.N.R. project "PREFERED".


\begin{bibdiv}
\begin{biblist}

\bib{abramowitz-stegun}{book}{
   author={Abramowitz, Milton},
   author={Stegun, Irene A.},
   title={Handbook of mathematical functions with formulas, graphs, and
   mathematical tables},
   series={National Bureau of Standards Applied Mathematics Series},
   volume={55},
   publisher={For sale by the Superintendent of Documents, U.S. Government
   Printing Office, Washington, D.C.},
   date={1964},
}

\bib{adams}{book}{
   author={Adams, Robert A.},
   title={Sobolev spaces},
   note={Pure and Applied Mathematics, Vol. 65},
   publisher={Academic Press [A subsidiary of Harcourt Brace Jovanovich,
   Publishers], New York-London},
   date={1975},
   pages={xviii+268},
}


\bib{bcmr}{article}{
   author={Brezis, Ha{\"{\i}}m},
   author={Cazenave, Thierry},
   author={Martel, Yvan},
   author={Ramiandrisoa, Arthur},
   title={Blow up for $u_t-\Delta u=g(u)$ revisited},
   journal={Adv. Differential Equations},
   volume={1},
   date={1996},
   number={1},
   pages={73--90},
   issn={1079-9389},
}

\bib{brezis-vazquez}{article}{
   author={Brezis, Haim},
   author={V{\'a}zquez, Juan Luis},
   title={Blow-up solutions of some nonlinear elliptic problems},
   journal={Rev. Mat. Univ. Complut. Madrid},
   volume={10},
   date={1997},
   number={2},
   pages={443--469},
   issn={0214-3577},
}

\bib{cabre}{article}
{
   author={Cabr{\'e}, Xavier},
   title={Regularity of minimizers of semilinear elliptic problems up to dimension four},
   date={2009},
}



\bib{CC04}{article}{
   author={Cabr{\'e}, Xavier},
   author={Capella, Antonio},
   title={Regularity of radial minimizers and extremal solutions of
   semilinear elliptic equations},
   journal={J. Funct. Anal.},
   volume={238},
   date={2006},
   number={2},
   pages={709--733},
   issn={0022-1236},
}

\bib{cabre-solamorales}{article}{
   author={Cabr{\'e}, Xavier},
   author={Sol{\`a}-Morales, Joan},
   title={Layer solutions in a half-space for boundary reactions},
   journal={Comm. Pure Appl. Math.},
   volume={58},
   date={2005},
   number={12},
   pages={1678--1732},
   issn={0010-3640},
}


\bib{cabre-tan}{article}
{
   author={Cabr{\'e}, Xavier},
   author={Tan, Jinggang},
   title={Positive solutions of nonlinear problems involving the square root of the Laplacian},
   date={2009},
}

\bib{Cabre-Sire}{article}
{
   author={Cabr{\'e}, Xavier},
   author={Sire, Y.},
   title={Nonlinear equations for fractional Laplacians I: Regularity, maximum principles,
and Hamiltonian estimates},
   date={2010},
}


\bib{cafS}{article}{
   author={Caffarelli, Luis},
   author={Silvestre, Luis},
   title={An extension problem related to the fractional Laplacian},
   journal={Comm. Partial Differential Equations},
   volume={32},
   date={2007},
   number={7-9},
   pages={1245--1260},
   issn={0360-5302},
}

\bib{caffarelli-salsa-silvestre}{article}{
   author={Caffarelli, Luis A.},
   author={Salsa, Sandro},
   author={Silvestre, Luis},
   title={Regularity estimates for the solution and the free boundary of the
   obstacle problem for the fractional Laplacian},
   journal={Invent. Math.},
   volume={171},
   date={2008},
   number={2},
   pages={425--461},
   issn={0020-9910},
}



\bib{crandall-rabinowitz}{article}{
   author={Crandall, Michael G.},
   author={Rabinowitz, Paul H.},
   title={Some continuation and variational methods for positive solutions
   of nonlinear elliptic eigenvalue problems},
   journal={Arch. Rational Mech. Anal.},
   volume={58},
   date={1975},
   number={3},
   pages={207--218},
   issn={0003-9527},
}

\bib{davila-dupaigne-montenegro}{article}{
   author={D{\'a}vila, Juan},
   author={Dupaigne, Louis},
   author={Montenegro, Marcelo},
   title={The extremal solution of a boundary reaction problem},
   journal={Commun. Pure Appl. Anal.},
   volume={7},
   date={2008},
   number={4},
   pages={795--817},
   issn={1534-0392},
}

\bib{FKS}{article}{
   author={Fabes, Eugene B.},
   author={Kenig, Carlos E.},
   author={Serapioni, Raul P.},
   title={The local regularity of solutions of degenerate elliptic
   equations},
   journal={Comm. Partial Differential Equations},
   volume={7},
   date={1982},
   number={1},
   pages={77--116},
   issn={0360-5302},
}

\bib{gelfand}{article}{
   author={Gel{\cprime}fand, I. M.},
   title={Some problems in the theory of quasilinear equations},
   journal={Amer. Math. Soc. Transl. (2)},
   volume={29},
   date={1963},
   pages={295--381},
}
\bib{gnn}{article}{
   author={Gidas, B.},
   author={Ni, Wei Ming},
   author={Nirenberg, L.},
   title={Symmetry and related properties via the maximum principle},
   journal={Comm. Math. Phys.},
   volume={68},
   date={1979},
   number={3},
   pages={209--243},
   issn={0010-3616},
}


\bib{joseph-lundgren}{article}{
   author={Joseph, D. D.},
   author={Lundgren, T. S.},
   title={Quasilinear Dirichlet problems driven by positive sources},
   journal={Arch. Rational Mech. Anal.},
   volume={49},
   date={1972/73},
   pages={241--269},
   issn={0003-9527},
}


\bib{lions}{article}{
   author={Lions, J.-L.},
   title={Th\'eor\`emes de trace et d'interpolation. I},
   journal={Ann. Scuola Norm. Sup. Pisa (3)},
   volume={13},
   date={1959},
   pages={389--403},
}

\bib{lions-magenes}{book}{
   author={Lions, J.-L.},
   author={Magenes, E.},
   title={Probl\`emes aux limites non homog\`enes et applications. Vol. 1},
   series={Travaux et Recherches Math\'ematiques, No. 17},
   publisher={Dunod},
   place={Paris},
   date={1968},
   pages={xx+372},
}

\bib{liouville}{article}{
   author={Liouville, R.},
   title={Sur l'{\'e}quation aux diff{\'e}rences partielles $d^2 log\lambda/du dv \pm \lambda/(2 a^2) = 0$},
   journal={Journal de Mathematiques Pures et Appliques},
   volume={18},
   date={1853},
   number={},
   pages={71--72},
}

\bib{mignot-puel}{article}{
   author={Mignot, Fulbert},
   author={Puel, Jean-Pierre},
   title={Sur une classe de probl\`emes non lin\'eaires avec non
   lin\'eairit\'e positive, croissante, convexe},
   journal={Comm. Partial Differential Equations},
   volume={5},
   date={1980},
   number={8},
   pages={791--836},
   issn={0360-5302},
}


\bib{nedev}{article}{
   author={Nedev, Gueorgui},
   title={Regularity of the extremal solution of semilinear elliptic
   equations},
   journal={C. R. Acad. Sci. Paris S\'er. I Math.},
   volume={330},
   date={2000},
   number={11},
   pages={997--1002},
   issn={0764-4442},
}


\bib{tartar}{book}{
   author={Tartar, Luc},
   title={An introduction to Sobolev spaces and interpolation spaces},
   series={Lecture Notes of the Unione Matematica Italiana},
   volume={3},
   publisher={Springer},
   place={Berlin},
   date={2007},
   pages={xxvi+218},
   isbn={978-3-540-71482-8},
   isbn={3-540-71482-0},
}




\end{biblist}
\end{bibdiv}
\end{document}